\documentclass[12pt]{article}
\usepackage{fullpage}

\usepackage{caption}
\usepackage{url}
\usepackage{amssymb,amsmath,amsthm}
\usepackage{color,xcolor,graphicx,tikz}
\usepackage[normalem]{ulem}
\usepackage{enumerate}
\usepackage{cleveref}

\newtheorem{theorem}{Theorem}[section]
\newtheorem{lemma}[theorem]{Lemma}
\newtheorem{fact}[theorem]{Fact}
\newtheorem{defn}[theorem]{Definition}
\newtheorem{cor}[theorem]{Corollary}

\newcommand{\paren}[1]{\left(#1\right)}

\newcommand{\floor}[1]{\left\lfloor{#1}\right\rfloor}
\newcommand{\set}[1]{\left\{#1\right\}}
\newcommand{\setsize}[1]{\left|#1\right|}

\newcommand{\mc}[1]{\mathcal{#1}}

\newcommand{\mdrho}{\rho^{m, \Delta}} % multigraph+degree bound
\newcommand{\merho}{\rho^{m, e}} % multigraph+edge bound
\newcommand{\sdrho}{\rho^{s, \Delta}} % simple graph+degree bounded
\newcommand{\dF}{\mc F^{\Delta}_D} % forest+degree bounded
\newcommand{\ndF}[1]{\mc F^{\Delta}_{#1}} % forest+degree bounded
\newcommand{\dFF}{(\dF, \mc F)} 
\newcommand{\ndFF}[1]{(\ndF{#1}, \mc F)} 

\newcommand{\eF}{\mc F^{e}_D} % forest+edges bounded
\newcommand{\neF}[1]{\mc F^{e}_{#1}} 
\newcommand{\eFF}{(\eF, \mc F)} 
\newcommand{\neFF}[1]{(\neF{#1}, \mc F)}

\title{Partition of Sparse Multigraphs into a Forest and a Forest with Restrictions}
\author{
{{Ilkyoo Choi}}\thanks{
\footnotesize {Hankuk University of Foreign Studies, Yongin-si, Gyeonggi-do, Republic of Korea.
 E-mail: \texttt {ilkyoo@hufs.ac.kr}
 and  Discrete Mathematics Group, Institute for Basic Science (IBS), Daejeon, Republic of Korea. 
 Research %%% of this author
is supported in part by the Hankuk University of Foreign Studies Research Fund and by the Institute for Basic Science (IBS-R029-C1).
}}
\and
{{Alexandr Kostochka}}\thanks{
\footnotesize {University of Illinois at Urbana--Champaign, Urbana, IL, USA.
 E-mail: \texttt {kostochk@illinois.edu}.
 Research %%% of this author
is supported in part by  NSF  Grant DMS-2153507 and by NSF RTG Grant DMS-1937241.
}}
\and{{Matthew P. Yancey}}\thanks{Institute for Defense Analyses - Center for Computing Sciences, Bowie, MD, USA. 
 E-mail: \texttt{mpyancey1@gmail.com}.
}}

\begin{document}
\maketitle

\begin{abstract} 
The following measure of sparsity of multigraphs refining the maximum average degree: For $a>0$ and an arbitrary real $b$,
a multigraph $H$ is \emph{$(a,b)$-sparse} if it is  loopless and for every $A\subseteq V(H)$ with $|A|\geq 2$,
the induced subgraph $H[A]$ has at most $a|A|+b$ edges.
Forests are exactly $(1,-1)$-sparse multigraphs. It is known that the vertex set of any $(2,-1)$-sparse multigraph can be partitioned into two parts each of which induces a forest. 

For a given parameter $D$ we study for which pairs $(a,b)$ every $(a,b)$-sparse multigraph $G$ admits
a vertex partition $(V_1, V_2)$ of $V(G)$ such that $G[V_1]$ and $G[V_2]$ are forests, and in addition either (i) $\Delta(G[V_1])\leq D$ or (ii) every component of $G[V_1]$ has at most $D$ edges.
We find exact bounds on $a$ and $b$ for both types of problems (i) and (ii). 
We also consider problems of type (i) in the class of simple graphs and find exact bounds for all $D\geq 2$.
%
% We prove that a $(7/4, -2)$-strictly sparse multigraph has a $(M, F)$-coloring, while there exist infinitely many $(7/4, -2)$-tight multigraphs do not have a $(M,F)$-coloring. 
\end{abstract}

\section{Introduction}

\subsection{History and results}
All multigraphs in this paper are loopless.
%, so below we always drop the word "loopless".
For a multigraph $G$, we use $V(G), E(G), |G|,
\Delta(G)$, and $\delta(G)$ to denote its vertex set, edge set, order, maximum degree, and minimum degree, respectively. 
A {\em $k$-vertex} is a vertex of degree $k$. 
A {\em forest} is an acyclic graph. 

We consider 
partitions of vertex sets of  multigraphs and simple graphs
into two forests. Recall that the {\em vertex arboricity} $va(G)$ of a  multigraph $G$ is the minimum number $k$ such that there is a vertex partition $(V_1,\ldots,V_k)$ of $V(G)$ such that the subgraph $G[V_i]$ of $G$ induced by $V_i$
has no cycles for each $1\leq i\leq k$. This parameter was introduced in 1968 by Chartrand, Kronk, and Wall~\cite{1968ChKrWa}, who called it {\em point-arboricity}.

Hakimi and Schmeichel~\cite{1989HaSc} proved that the decision problem whether a given planar graph has vertex arboricity at most $2$ is NP-complete. In a series of papers~\cite{2009BoIv,2012ChRaWa,2012HuShWa,2013HuWa,2008RaWa,1971Stein,2024WaWaLi} it was proved  that  broad subclasses of planar graphs have  vertex arboricity 
 at most $2$. Mostly, these classes do not contain graphs with  high average degree. Other sufficient conditions for $va(G)\leq 2$ were proved for more general classes of sparse graphs. 
Meaningful measures of sparseness of a (multi)graph $G$ are maximum degree,{\em maximum average degree}, 
$mad(G)=\max_{H\subseteq G}\frac{2|E(H)|}{|V(H)|}$, and its refinement, 
$(a,b)$-sparseness.
For $a>0$ and an arbitrary real $b$, call
a multigraph $H$  $(a,b)$-{\em sparse} if  for every $A\subseteq V(H)$
with $|A|\geq 2$, the induced subgraph $H[A]$ has at most $a|A|+b$ edges.
By definition, forests are exactly $(1,-1)$-sparse multigraphs.

 % One measure of sparseness is maximum degree. 
 A result of Borodin~\cite{1976Borodin} from 1976, later reproved by Bollob\' as and Manvel~\cite{1979BoMa} and Catlin and Lai~\cite{1995CaLa}, implies that every simple  graph $G$ with $\Delta(G)\leq 4$ not containing $K_5$  satisfies $va(G)\leq 2$. A corollary of this result  is that every graph $G$  with $mad(G)\leq 4$ (i.e. $(2,0)$-sparse) not containing $K_5$
 satisfies $va(G)\leq 2$. 

 A number of papers also consider partitions of multigraphs into two forests in which one or both of the forests has restrictions.
 For an integer $D$, an {\em $\dFF$-coloring} of a multigraph $G$ is a partition $(M, F)$ of $V(G)$ such that   $G[M]$ is a forest with maximum degree at most $D$ and $G[F]$ is a forest. 
  
  The case of $D=0$ (i.e. $M$ is an independent set) of $\dFF$-coloring has  attracted much attention. It was proved (see e.g.~\cite{2013BrBrKl,2017DrMoPi,2006YaYu}) that deciding whether a graph has an $\ndFF{0}$-coloring  is an NP-complete problem even in some  classes of graphs with moderate average degree. In particular, Dross,  Montassier, and  Pinlou~\cite{2017DrMoPi} proved this for the class of planar graphs, and Yang and  Yuan~\cite{2006YaYu} proved this for the class of graphs with maximum degree at most $4$.

On the other hand, for several versions of sparseness it was proved that ``sparse'' graphs admit an $\ndFF{0}$-coloring. 
Borodin and Glebov~\cite{2001BoGl} proved that every planar graph of girth at least 5
has an $\ndFF{0}$-coloring.  Dross,  Montassier, and  Pinlou~\cite{2017DrMoPi} conjectured that this holds already for planar graphs of girth at least $4$.

The above mentioned result by Borodin~\cite{1976Borodin}, Bollob\' as and Manvel~\cite{1979BoMa}, and Catlin and Lai~\cite{1995CaLa} yields that every simple graph $G$ with $\Delta(G)\leq 3$ not containing $K_4$ has an $\ndFF{0}$-coloring.

 Cranston and Yancey~\cite{2020CrYa} proved that every $(1.5,0.5)$-sparse multigraph  containing neither $K_4$ nor a special $7$-vertex graph called the {\em Moser Spindle} has an $\ndFF{0}$-coloring. They proved further that there is a finite set $\cal H$ of graphs such that every $(1.6,0.8)$-sparse simple graph not containing a subgraph in $\cal H$ also has an $\ndFF{0}$-coloring. 
Both bounds are exact.

In this paper we consider $\dFF$-colorings for $D\geq 1.$ 
Our results on such colorings are:

\begin{theorem}\label{thm:multi-degreeS}
For each $D\geq 1$, every $\paren{\frac{4D+3}{2(D+1)},\frac{1}{2(D+1)}}$-sparse
   multigraph has an $\dFF$-coloring.
\end{theorem}

\begin{theorem}\label{thm:simple-degreeS}
For each $D\geq 2$, every $\paren{\frac{6D+5}{3(D+1)},\frac{2}{3(D+1)}}$-sparse
   simple graph has an $\dFF$-coloring.
\end{theorem}

Both bounds are exact. Note that when $D=1$, the corresponding sparseness  of Theorem~\ref{thm:simple-degreeS} does not hold as there are more sparse ``critical'' graphs.

Dross,  Montassier, and  Pinlou~\cite{2017DrMoPi} showed that planar graphs of girth at least $4$ admit an $\ndFF{5}$-coloring.
This was improved by Feghali and \v{S}\'amal~\cite{2024FeSa} to an $\ndFF{3}$-coloring.
Liu and Wang~\cite{2022LiWa} showed that planar graphs of girth at least $4$ and no chorded $6$-cycle admit an $\ndFF{2}$-coloring.
We use the above result to give Corollary \ref{LW-reproved}, which is an independent proof of Liu and Wang's result.

\bigskip
We also consider a similar problem with stronger restrictions on one of the forests.
An {\em $\eFF$-coloring} of a multigraph $G$ is a partition $(M, F)$ of $V(G)$ such that $G[M]$ is a forest in which every component has at most $D$ edges and $G[F]$ is a forest. For $D\in \{0,1\}$, an $\eFF$-coloring is simply an
 $\dFF$-coloring, but for $D\geq 2$, an $\eFF$-coloring is significantly more restrictive than an $\dFF$-coloring.

 Cranston and Yancey~\cite{2021CrYa} considered the similar problem when not only the components of $G[M]$ are bounded, but also $F$ must be an independent set. They found an exact bound on the maximum average degree of a (simple) graph that guarantees the existence of such a partition.
  
Our  result on $\eFF$-coloring is:

\begin{theorem}\label{thm:multi-orderS}
(i) For each odd $D\geq 1$, every $\paren{\frac{4D+3}{2D+2},\frac{1}{2(D+1)}}$-sparse
   multigraph has an $\eFF$-coloring.

(ii) For each even $D\geq 2$, every $\paren{\frac{4D+1}{2D+1},\frac{1}{2D+1}}$-sparse
   multigraph has an $\eFF$-coloring.  
\end{theorem}

Both parts of Theorem~\ref{thm:multi-orderS} are sharp. 
Note that  the bound on sparseness for odd $D$ is the same as in Theorem~\ref{thm:multi-degreeS}, but
for even $D$ it is different.

At the Barbados 2019 Graph Theory Workshop, Hendrey, Norin, and Wood asked for the largest value of $f_{a,b}$ such that any graph $G$ with $mad(G) < f_{a,b}$ has a vertex  partition $(A, B)$ of $V(G)$ such that $mad(G[A]) < a$ and $mad(G[B]) < b$.
%The answer to this question can change by allowing multigraphs or excluding a finite set of graphs (such as complete graphs).
Observe that $mad(G[A]) < \frac{2D}{D+1}$ if and only if every component of $G[A]$ has at most $D-1$ edges. 
%$ \in \neF{D-1}$. %\mathcal{F}_{D-1}^e$.
Thus, Theorem \ref{thm:multi-orderS} answers the analogous question for multigraphs when $b=2$ and $a < 2$.

\subsection{Proof ideas and structure of the paper}

We will prove our results in terms of the sizes of critical multigraphs, which are formally stronger than
Theorems~\ref{thm:multi-degreeS}--\ref{thm:multi-orderS}.
We call  a multigraph  $\dFF$-{\em critical} (respectively, $\eFF$-{\em critical}) if it does not admit an $\dFF$-coloring (respectively, an $\eFF$-coloring), but every proper subgraph does. In these terms, the counterparts of Theorems~\ref{thm:multi-degreeS}--\ref{thm:multi-orderS} are as follows.

\begin{theorem}\label{thm:multi-degreeSc}
For each $D\geq 1$, every $\dFF$-{critical} multigraph $G$ on at least three vertices satisfies $|E(G)|\geq \frac{(4D+3)|V(G)|+2}{2(D+1)}$.
\end{theorem}

\begin{theorem}\label{thm:simple-degreeSc}
For each $D\geq 2$, every $\dFF$-{critical} simple graph $G$ on at least three vertices satisfies $|E(G)|\geq \frac{(6D+5)|V(G)|+3}{3(D+1)}$.
\end{theorem}

\begin{theorem}\label{thm:multi-orderSc}
(i) For each  odd $D\geq 1$, every $\eFF$-{critical} multigraph $G$ on at least three vertices satisfies
$|E(G)|\geq \frac{(4D+3)|V(G)|+2}{2(D+1)}$.

(ii) For each even $D\geq 2$,  every $\eFF$-{critical} multigraph $G$ on at least three vertices satisfies
$|E(G)|\geq \frac{(4D+1)|V(G)|+2}{2D+1}$.
\end{theorem}

Theorem \ref{thm:simple-degreeSc} provides a short proof of Liu and Wang's result.

\begin{cor}[\cite{2022LiWa}]\label{LW-reproved}
Planar graphs of girth at least $4$ and no chorded $6$-cycle admit a $\ndFF{2}$-coloring.
\end{cor}
\begin{proof}
Suppose $G$ is a simple, planar graph with neither a $3$-cycle nor a chorded $6$-cycle that is $\ndFF{2}$-critical.
If $G$ has a vertex $v$ of degree at most $2$, then an $\ndFF{2}$-coloring of $G-v$ can be extended to $G$ by adding $v$ to $M$ if all neighbors of $v$ are in $F$, and adding $v$ to $F$ otherwise.
This contradicts the criticality of $G$, so $\delta(G) \geq 3$.
Because $G$ is a simple graph without $3$-cycles, chorded $6$-cycles, or $2$-vertices, there are no adjacent faces with length $4$.
We now apply discharging.
Give each face an initial charge equal to its length. 
Each face with length at least $5$ gives charge $1/9$ to each face that shares an edge with it (with multiplicity, should it share more than one edge).
By the above, each face ends with charge at least $40/9$.
Let $n,e,f$ denote the number of vertices, edges, and faces of $G$.
By Euler's formula and the above we have $n-e+f=2$ and $f \leq 9e/20$, which implies $e \leq \frac{20}{11}(n-2) < 1.82n$.
This contradicts Theorem \ref{thm:simple-degreeSc}, which states that $e \geq \frac{17n+3}{9} > 1.88n$.    
\end{proof}

To prove Theorems~\ref{thm:multi-degreeSc}--\ref{thm:multi-orderSc}, it will be convenient to use the language of
{\em potentials}. For example, in Theorem~\ref{thm:multi-degreeSc} every nonempty $A\subseteq V(G)$ has potential
$\rho(A)=(4D+3)|A|-(2D+2)|E(G[A])|$ and the statement in these terms is that for every $\dFF$-{critical} multigraph $G$,
$\rho(V(G))\leq -2$. Moreover, to facilitate induction, it will be convenient to introduce concepts such as capacities and weights of  vertices in order to modify potentials of these vertices. 
%depending on these capacities and weights.

In the next section we present constructions showing sharpness of all our bounds. In \Cref{sec-premise} we set up the proofs of lower bounds in (refined versions of) Theorems~\ref{thm:multi-degreeSc}--\ref{thm:multi-orderSc}
and discuss the common steps in their proofs.
We gather the gap lemmas in Section~\ref{sec-gaplems} and collect common reducible configurations in \Cref{sec-common-reducibles}.
% In Sections 4, \Cref{sec-dFF-coloring}, and~\Cref{sec-eFF-coloring} we finish the proofs of
% Theorems~\ref{thm:multi-degreeSc}, \ref{thm:simple-degreeSc}, and \ref{thm:multi-orderSc}, respectively.
In \Cref{sec-dFF-coloring} we prove \Cref{thm:multi-degreeSc} when $D\geq 2$ and \Cref{thm:simple-degreeSc} at the same time. 
The case when $D=1$ for \Cref{thm:multi-degreeSc} is postponed to \Cref{sec-eFF-coloring}. 
In \Cref{sec-eFF-coloring} we prove \Cref{thm:multi-orderSc}; the proof for odd $D$ needs an additional argument so it is in its own subsection.

%$\ndFF{1}$-coloring $\neFF{1}$-coloring

% {\MY We had a construction of 9/5 for simple graphs, but I can not remember it.  Do you recall it?}
% {\IC yes. figure added.}

% \begin{figure}[ht]
% \begin{center}
%   \includegraphics[scale=0.3]{figs/multi_example.png} 
%   \includegraphics[scale=0.3]{figs/simple_example.png} 
% % \caption{ }\label{fig:example_1}
% \end{center}
% \end{figure}

%We consider a multiedge to be a $2$-cycle, hence the endpoints of any multiedge must be in different color classes. 

\section{Constructions}\label{constructions section}

\def\xspace{1}
\def\yspace{0.75}

\newcommand{\twoFunGadget}[2]{
\begin{scope}[rotate=#2]
    \draw (#1)++(20:1) node[graynode] (w1){};
    \draw (#1)++(-20:1) node[graynode] (w2){};
    \foreach \x [count=\xi from 1] in {1,...,2}                  \draw  (#1)--(w\xi);
    \draw (w1) to [bend left] (w2);
    \draw (w1) to [bend right] (w2);
\end{scope}
}

\def\figMultiDegree{
% vertices
\draw (0,0) node [graynode,label=below:$x_0$] (c){};
 \foreach \i [count=\xi from 1] in {1,...,7}
     \draw (c)++(\xi*1.5,0) node[graynode] (c\xi){};
\foreach \i [count=\xi from 1] in {1,...,4}
     \draw (c)++(\xi*1.5,0) node[graynode, label=below:$x_\xi$] (){};
\draw (c)++(5*1.5,0) node[graynode, label=below:$x_{2k-1}$] (){};
\draw (c)++(6*1.5,0) node[graynode, label=below:$x_{2k}$] (){};
\draw (c)++(7*1.5,0) node[graynode, label=below:$x_{2k+1}$] (){};

% edges
 \foreach \i/\j in {c/c1, c2/c3, c6/c7}{
    \draw (\i) to [out=20, in=180-20] (\j);
    \draw (\i) to [out=-20, in=180+20] (\j);
 }
 \foreach \i/\j in {c1/c2, c3/c4, c5/c6}{
    \draw (\i) to (\j);
 }
  \draw (c4)++(1.5/2,0) node[](){$\cdots$};

% gadgets
 \foreach \i in {2,4,6}{
    \twoFunGadget{c\i}{90+45}
    \draw (c\i) ++ (0,1) node[](){$\ldots$};
    \draw (c\i) ++ (0,1.5) node[](){$D$};
    \twoFunGadget{c\i}{90-45}
 }
\twoFunGadget{c}{90+90}
\draw (c) ++ (90+45-2.5:1) node[]() {\reflectbox{$\ddots$}};
\draw (c) ++ (90+45+2.5:1.75) node[](){$D+1$};
\twoFunGadget{c}{90}

\twoFunGadget{c7}{90-90}
\draw (c7) ++ (90-45+2.5:1) node[](){$\ddots$};
\draw (c7) ++ (90-45-2.5:1.75) node[](){$D+1$};
\twoFunGadget{c7}{90}

% \foreach \c in {1,...,9}
%     \draw[thick,dotted] (c) to (u\c);
% \draw (c)++(0,-3) node[graynode,label=$z$](n){};
% \draw (n) to [out=30, in=180-30,looseness=30](n) ;
}

\def\Xx{1}

\def\figMultiEdges{
% base
\foreach\z in {0,5,13}{
 % vertices
\draw (0,0) node [graynode] (c){};
 \foreach \i in {0,\Xx}{
    \foreach \j in {0,\Xx}{
        \draw (c)++(\i+\z,\j) node[graynode] (x\z\i\j){};
     }
 }
 % edges
 \foreach \i in {0,\Xx}{
 \foreach \j in {0,\Xx}{
    \foreach \ii in {0,\Xx}{
    \foreach \jj in {0,\Xx}{
        \draw (x\z\i\j) to (x\z\ii\jj);
 }}}}

 % gadgets
 \twoFunGadget{x\z\Xx\Xx}{-45+90+45}
 \draw (x\z\Xx\Xx) ++ (-45+90+2.5:1) node[](){$\ddots$};
 \draw (x\z\Xx\Xx) ++ (-20+90:1.75) node[](){$\floor{\frac{D+1}{2}}$};
 \twoFunGadget{x\z\Xx\Xx}{-45+90-45}

 \twoFunGadget{x\z0\Xx}{135+45}
 \draw (x\z0\Xx) ++ (135-2.5:1) node[](){\reflectbox{$\ddots$}};
 \draw (x\z0\Xx) ++ (20+90:1.75) node[](){$\floor{\frac{D+1}{2}}$};
 \twoFunGadget{x\z0\Xx}{135-45}

 \twoFunGadget{x\z00}{45+180+45}
 \draw (x\z00) ++ (45+180-2.5:1) node[](){$\ddots$};
 \twoFunGadget{x\z00}{45+180-45}

 \twoFunGadget{x\z\Xx0}{45+270+45}
 \draw (x\z\Xx0) ++ (45+270+2.5:1) node[](){\reflectbox{$\ddots$}};
 \twoFunGadget{x\z\Xx0}{45+270-45}
}
% labels for bottom gadgets
\foreach\z in {5,13}{
 \draw (x\z00) ++ (45+180+20:1.75) node[](){$\floor{\frac{D-2}{2}}$};
}
 \draw (x000) ++ (45+180+20:1.75) node[](){$\floor{\frac{D}{2}}$};

\foreach\z in {0,5}{
 \draw (x\z\Xx0) ++ (45+270-20:1.75) node[](){$\floor{\frac{D-1}{2}}$};
}
 \draw (x13\Xx0) ++ (45+270-20:1.75) node[](){$\floor{\frac{D+1}{2}}$};

% between gadgets
 \draw (3-0.35,\Xx/2+0.03) node[graynode] (x1){};
 \draw (3+0.35, \Xx/2+0.03) node[graynode] (x2){};
 \draw (x1) to [bend left] (x2);
 \draw (x1) to [bend right] (x2);
 \draw (x0\Xx0) to [out=45, in=180](x1);
 \draw (x2) to [out=0, in=135] (x500);

 \draw (8-0.35,\Xx/2+0.03) node[graynode] (x1){};
 \draw (8+0.35, \Xx/2+0.03) node[graynode] (x2){};
 \draw (x1) to [bend left] (x2);
 \draw (x1) to [bend right] (x2);
 \draw (x5\Xx0) to [out=45, in=180](x1);
 \draw (x2) -- ++ (0.7,0);
 
 \draw (11-0.35,\Xx/2+0.03) node[graynode] (x1){};
 \draw (11+0.35, \Xx/2+0.03) node[graynode] (x2){};
 \draw (x1) to [bend left] (x2);
 \draw (x1) to [bend right] (x2);
 \draw (x1300) to [out=90+45, in=0](x2);
 \draw (x1) -- ++ (-0.7,0);

 \draw (19/2, \Xx/2+0.03) node[] () {$\cdots$};
% vertex labels
 \draw (x1310) node[graynode,label=below right:$v_k$](){};
 \draw (x510) node[graynode,label=below right:$v_2$](){};
 \draw (x010) node[graynode,label=below right:$v_1$](){};

% \foreach \c in {1,...,9}
%     \draw[thick,dotted] (c) to (u\c);
% \draw (c)++(0,-3) node[graynode,label=$z$](n){};
% \draw (n) to [out=30, in=180-30,looseness=30](n) ;
}

\newcommand{\threeFunGadget}[2]{
\begin{scope}[rotate=#2]
    \draw (#1)++(20:1) node[graynode] (w1){};
    \draw (#1)++(-20:1) node[graynode] (w2){};
    \draw (#1)++(0:1.5) node[graynode] (w3){};
    \foreach \x [count=\xi from 1] in {1,...,3}                  \draw  (#1)--(w\xi);
    \draw (w1) -- (w2) -- (w3) -- (w1);
\end{scope}
}

\def\figSimpleDegrees{
% vertices
\draw (0,0) node [graynode,label=below:$z_1$] (z1){};
 \draw (z1) 
 -- ++ (180-15:2) node[graynode,label=right:$x_1$](x1){} ;
 \draw (z1) 
 -- ++ (180+15:2) node[graynode,label=right:$y_1$](y1){} -- (x1);

 \draw (z1) 
 -- ++ (15:2) node[graynode,label=above:$x_2$](x2){} ;
 \draw (z1) 
 -- ++ (-15:2) node[graynode,label=below:$y_2$](y2){} -- (x2) -- ++ (-15:2) node[graynode,label=below:$z_2$](z2){} -- (y2);

 \draw (z2) -- ++ (15:1);
 \draw (z2) -- ++ (-15:1);

 \draw (6.6,0) node[graynode,label=below:$z_{k-1}$] (zk-1){} -- ++ (15:2) node[graynode,label=above:$x_k$](xk){} ;
 \draw (zk-1) 
 -- ++ (-15:2) node[graynode,label=below:$y_k$](yk){} -- (xk) -- ++ (-15:2) node[graynode,label=below:$z_k$](zk){} -- (yk);

 \draw (zk-1) -- ++ (180+15:1);
 \draw (zk-1) -- ++ (180-15:1);

% dots
  \draw (z2) ++ (1.4,0) node[](){$\cdots$};

% gadgets
    \threeFunGadget{x1}{135+45}
    \draw (x1) ++ (135-2.5:1) node[](){\reflectbox{$\ddots$}};
    \draw (x1) ++ (135:1.75) node[](){$D+1$};
    \threeFunGadget{x1}{135-45}

    \threeFunGadget{y1}{-135+45}
    \draw (y1) ++ (-135-2.5:1) node[](){$\ddots$};
    \draw (y1) ++ (-135:1.75) node[](){$D+1$};
    \threeFunGadget{y1}{-135-45}

    \threeFunGadget{z1}{90+35}
    \draw (z1) ++ (90-2.5:1.2) node[](){$\cdots$};
    \draw (z1) ++ (90:1.75) node[](){$D$};
    \threeFunGadget{z1}{90-35}

    \threeFunGadget{z2}{90+35}
    \draw (z2) ++ (90-2.5:1.2) node[](){$\cdots$};
    \draw (z2) ++ (90:1.75) node[](){$D$};
    \threeFunGadget{z2}{90-35}

    \threeFunGadget{zk-1}{90+35}
    \draw (zk-1) ++ (90-2.5:1.2) node[](){$\cdots$};
    \draw (zk-1) ++ (90:1.75) node[](){$D$};
    \threeFunGadget{zk-1}{90-35}

    \threeFunGadget{zk}{0+45}
    \draw (zk) ++ (1,0.1) node[](){$\vdots$};
    \draw (zk) ++ (0:1.75) node[](){$D+1$};
    \threeFunGadget{zk}{0-45}

% \foreach \c in {1,...,9}
%     \draw[thick,dotted] (c) to (u\c);
% \draw (c)++(0,-3) node[graynode,label=$z$](n){};
% \draw (n) to [out=30, in=180-30,looseness=30](n) ;
}

In this section, we present families of graphs demonstrating the sharpness of Theorems~\ref{thm:multi-degreeSc}, \ref{thm:simple-degreeSc}, and \ref{thm:multi-orderSc}.

The operation of ``adding a $2$-fundamental gadget'' to a vertex $v$ means to add two vertices $x$ and $y$ and  four edges $vx, vy, xy, xy$. 
Note that $x$ and $y$ cannot be in the same color class in an $\dFF$-coloring or an $\eFF$-coloring. 
In particular, given either a $\dFF$-coloring or a $\eFF$-coloring  $(M, F)$ of the graph, if $D+1$ $2$-fundamental gadgets have been added to $v$, then $v\in F$ and if $D$ $2$-fundamental gadgets have been added to $v$ and $v\in M$, then the neighbors of $v$ not in those $D$ $2$-fundamental gadgets must be in $F$. 

\noindent
{\bf Sharpness example for \Cref{thm:multi-degreeSc}: multigraphs and $\dFF$-coloring}

For an integer $k \geq 0$, construct a multigraph $H^{m, \Delta}_k(D)$ as follows.
Start with $2k+2$ vertices ${x_0, x_1, \ldots, x_{2k+1}}$ where $g(x_i)=\begin{cases}
    D+1 & \mbox{if $i\in\set{0, 2k+1}$,}\\
    D & \mbox{if $i$ is even and $i\in\set{1, \ldots, 2k}$,}\\
    0 & \mbox{if $i$ is odd and $i\in\set{1, \ldots, 2k}$.}
\end{cases}$

For each $i\in\set{0, \ldots, 2k+1}$, add $g(x_i)$ $2$-fundamental gadgets to $x_i$.
Note that a total of $Dk+2(D+1)$ $2$-fundamental gadgets have been introduced. 
For each $i \in \set{0, \ldots, 2k}$, $x_ix_{i+1}$ is a multi-edge if $i$ is even and a simple edge if $i$ is odd. 
See \Cref{fig-multi-degree}.
Observe that $H^{m, \Delta}_k(D)$ has $2k+2+2(Dk+2D+2)$ vertices and $3k+2+4(Dk+2D+2)$ edges; in other words, $H^{m, \Delta}_k(D)$ has $\frac{(4D+3)\setsize{V(H^{m, \Delta}_k(D))}+2}{2(D+1)}$ edges.  
% Observe that $\mdrho(H^{m, \Delta}_k(D)) = 2(2D+1)+k(4D+3)+k(2D+3)-(3k+2)(2D+2)=-2$ for all $k$.

We prove that $H^{m, \Delta}_k(D)$ has no $\dFF$-coloring by contradiction; suppose there is an $\dFF$-coloring $(M, F)$.
If $x_{2i} \in F$, then by its incident multi-edge  $x_{2i+1}$ must be in $M$.
If $x_{2i+1} \in M$ and $i < k$, then because $D$ $2$-fundamental gadgets have been added to $x_{2i+2}$, 
% $c(x_{2i+2}) = 1$ 
we have $x_{2i+2} \in F$.
As $D+1$ $2$-fundamental gadgets have been added to $x_0$, 
% $c(x_0) = 0$, 
we have $x_0 \in F$.
Applying the above logic for $i\in\set{0,1,\ldots,k-1}$ we get that $x_{2k} \in F$.
As $D+1$ $2$-fundamental gadgets have been added to $x_{2k+1}$, 
% $c(x_{2k+1}) = 0$, 
we have $x_{2k+1} \in F$.
But then $F$ contains a multi-edge $x_{2k}x_{2k+1}$, which is a contradiction.

%This construction is not unique.
%Consider $V = \{u_1, u_2, u_3, v_1, v_2, v_3\}$, where $c(u_i) = 0$ and $c(v_i) = 2$.
%Let $u_iv_i$ be a parallel edge for each $i$ and add edges $v_1v_2, v_2v_3$.

\begin{figure}
\centering
\begin{tikzpicture}
[scale=1,auto=left, 
blacknode/.style={circle,draw,fill=black,minimum size = 6pt,inner sep=0pt}, 
graynode/.style={circle,draw,fill=gray!30,minimum size = 6pt,inner sep=0pt}
]
 \begin{scope}[xshift=0cm]\figMultiDegree\end{scope}
\end{tikzpicture}
\caption{Sharpness example for \Cref{thm:multi-degreeSc}.}
\label{fig-multi-degree}
\end{figure}

\noindent
{\bf Sharpness example for \Cref{thm:multi-orderSc}: multigraphs and $\eFF$-coloring}

An {\em $M^*$-gadget} on $v_1$ is a complete graph on four vertices $v_1, v_2, v_3, v_4$ where every vertex has $\floor{\frac{D}{2}}$ $2$-fundamental gadgets, and if $D$ is even then remove a $2$-fundamental gadget from $v_1$ and if $D$ is odd then add an additional $2$-fundamental gadget on each of $v_3$ and $v_4$. 
Note that an $M^*$-gadget has $4D+2$ vertices and $8D+2$ edges when $D$ is even and $4D+4$ vertices and $8D+6$ edges when $D$ is odd. 
An {\em $(M^*, M^*)$-gadget} from $u$ to $v$ is obtained by starting with an $M^*$-gadget on $v$, removing a $2$-fundamental gadget on some vertex $w\neq v$, and adding a $2$-cycle $t_1t_2$ and two more edges $t_1u,t_2w$. 

Construct a multigraph $H^{m,e}_k(D)$ as follows. 
Start with $k$ vertices $v_1, \ldots, v_k$ where there is an $M^*$-gadget on $v_1$ and an $(M^*, M^*)$-gadget from $v_i$ to $v_{i+1}$ for each $i\in\set{1, \ldots, k-1}$, and add a $2$-fundamental gadget to $v_k$. 
See \Cref{fig-multi-edges}.
Observe that $H_k^{m, e}(D)$ has $(4D+4)k+2$ vertices and $(8D+6)k+4$ edges when $D$ is even and $(4D+2)k+2$ vertices and $(8D+2)k+4$ edges when $D$ is odd; in other words, $H_k^{m, e}(D)$ has $\frac{(4D+1)\setsize{V(H_k^{m, e}(D))}+2}{2D+1}$ edges when $D$ is even and $\frac{(4D+3)\setsize{V(H_k^{m, e}(D))}+2}{2D+2}$ when $D$ is odd. 

% The capacity or weight of a vertex contributes $2$ to the potential in prior sections, which matches the contribution of a fundamental gadget.
% In this setting there exists a gadget that contributes $D/2 + 2$ to the weight while only contributing $D+3$ to the potential, which motivates the modified potential function when $c(v) \leq (D-2)/2 = D+1-(D/2+2)$.
% The gadget is as follows: to add $D/2+2$ weight to a vertex $u$, add $u'$ and $M^*$ gadget on $u'$ with all fundamental gadgets on $u'$ removed.  

% \begin{figure}[ht]
% \begin{center}
%   \includegraphics[scale=0.04]{figs/d24k2.pdf} 
%  \caption{The graphs $H^{m,e}_3(2)$ and $H^{m,e}_3(4)$. }\label{fig:hme tightness}
% \end{center}
% \end{figure}

% The $M^*$-gadget has 
% $
% 4\frac{D}{2}-1=2D-1 $ fundamental gadgets. 
% Note that adding a fundamental gadget increases potential by 
% $
% 2(4D+1)-4(2D+1)=-2.$
% Therefore, the potential of the $M^*$-gadget is 
% $
%      4(4D+1)-6(2D+1)-2(2D-1)
%      = 0.   $
% The $(M^*, M^*)$-gadget has exactly one more vertex than the $M^*$-gadget and the same number of edges. 
% Since $v_1, \ldots, v_{k-1}$ are in two gadgets, $\merho(H^{m,e}_k(D))$ is the same as the potential of a fundamental gadget, which is $-2$ for all $D\geq1$. 

We claim that $H^{m,e}_k(D)$ has no $\eFF$-coloring $(M, F)$. 
We remark that in every copy of $K_4$, there must be two vertices in each of $M$ and $F$.
The $M^*$-gadget on $v_1$ forces $v_1$ in $M$; moreover, all neighbors of $v_1$ not in the $M^*$-gadget must be in $F$. 
For $i\in\set{1, \ldots, k-1}$, since $v_i$ is in $M$ the $(M^*, M^*)$-gadget forces $v_{i+1}$ to be in $M$.
Yet, the component containing $v_k$ has $D+1$ edges, which is a contradiction, so $H^{m,e}_k(D)$ has no $\eFF$-coloring. 
Note that $D+1=\min\set{\frac{D}{2}+\frac{D}{2}+1,\frac{D-1}{2}+\frac{D+1}{2}+1}$.

% Let $F_D$ denote the class of weighted forests $H$ such that for every component $C$ of $H$, 

% An $M^*$-gadget is a $K_4$ with vertices $v_1, v_2, v_3, v_4$, where vertex $v_i$ has $m_i$ fundamental gadgets, with $m_1 = \lfloor (D-1)/2 \rfloor$, $m_2 = \lfloor D/2 \rfloor$, and $m_3 = m_4 = \lfloor (D+1)/2 \rfloor$.

% \begin{claim}
%     In any $\eFF$-coloring $(M, F)$ of an $M^*$-gadget has a component $C'$ of $M$ such that $v_1 \in C$ and $\sum_{u\in V(C')}(1+w(u)) = D+1$.
% \end{claim}
% \begin{proof}
%     Because $v_1, v_2, v_3, v_4$ induce a $K_4$, exactly two of them are colored $M$.
%     Let $C'$ be the component of $M$ that intersects $v_1, v_2, v_3, v_4$.

%     Observe that for any integer exactly one of $a$ and $a+1$ is odd, and therefore $\lfloor a/2 \rfloor + \lfloor (a+1)/2 \rfloor = a/2 + (a+1)/2 - 1/2 = a$.    
%     Therefore $\sum_{u\in V(C')}(1+w(u)) \geq 1 + \lfloor (D-1)/2 \rfloor + 1 +  \lfloor D/2 \rfloor = D+1$.
%     Moreover, if $v_1 \notin C'$, then $\sum_{u\in V(C')}(1+w(u)) \geq 1 + \lfloor D/2 \rfloor + 1 +  \lfloor (D+1)/2 \rfloor = D+2$, which contradicts the $\eFF$-coloring.
% \end{proof}

%\begin{theorem}\label{m1}
  % Let $D\geq 2$.  If $G$ is a $\eFF$-critical graph, then $\merho(G)\leq -2$.
%\end{theorem}

\begin{figure}
\centering
\begin{tikzpicture}
[scale=1,auto=left, 
blacknode/.style={circle,draw,fill=black,minimum size = 6pt,inner sep=0pt}, 
graynode/.style={circle,draw,fill=gray!30,minimum size = 6pt,inner sep=0pt}
]
 \begin{scope}[xshift=0cm]\figMultiEdges\end{scope}
\end{tikzpicture}
\caption{Sharpness example for \Cref{thm:multi-orderSc}.}
\label{fig-multi-edges}
\end{figure}

\noindent
{\bf Sharpness example for \Cref{thm:simple-degreeSc}: simple graphs and $\dFF$-coloring}

The operation of ``adding a $3$-fundamental gadget'' to a vertex $v$ means to introduce a completely new $3$-cycle $xyz$ and add edges $vx,vy,vz$. 
(This is the same as introducing a new $K_4$ and identifying $v$ with a vertex in the new $K_4$.)
Note that $x,y,z$ cannot all be in the same color class in an $\dFF$-coloring (or an $\eFF$-coloring), so some vertex in $x,y,z$ must be in the same color class with $v$. 
In particular, given a ($\dFF$- or $\eFF$-)coloring  $(M, F)$ of the graph, if $D+1$ $3$-fundamental gadgets have been added to $v$, then $v\in F$ and if $D$ $3$-fundamental gadgets have been added to $v$ and $v\in M$, then the neighbors of $v$ not in those $D$ $3$-fundamental gadgets must be in $F$. 

Construct a (simple) graph $H^{s,\Delta}_k(D)$ as follows. 
Start with vertices $\set{x_i, y_i, z_i: i\in\set{1, \ldots, k}}$ where $x_i, y_i, z_i$ is a $3$-cycle for all $i$ and $z_ix_{i+1}, z_iy_{i+1}$ are edges for $i\in\set{1, \ldots, k-1}$.
Let $g(v)=\begin{cases}
    D+1 & \mbox{if $v\in\set{x_1, y_1, z_k}$,}\\
    D & \mbox{if $v\in\set{z_1, \ldots, z_{k-1}}$,}\\
    0 & \mbox{otherwise.}
\end{cases}$
% See Figure~?%\ref{TODO}.
% IC: I was going to add a figure afterwards. 
For each $v\in\set{x_1, y_1, z_1, \ldots, z_{k}}$, add $g(v)$ $3$-fundamental gadgets to $v$. 
Note that a total of $D(k-1)+3(D+1)$ $3$-fundamantal gadgets have been introduced. 
% Now, $\sdrho(H^{s,\Delta}_k(D))=3(3D+2)+(k-1)(3D+5)+(2k-2)(6D+5)-(5k-2)(3D+3)=-3$.
Observe that $H_k^{s, \Delta}(D)$ has $3k+3(Dk+2D+3)$ vertices and $5k-2+6(Dk+2D+3)$ edges; in other words, $H_k^{s, \Delta}(D)$ has $\frac{(6D+5)\setsize{V(H_k^{s, \Delta}(D))}+3}{3D+3}$ edges. 

We claim $H^{s,\Delta}_k(D)$ has no $\dFF$-coloring: suppose to the contrary an $\dFF$-coloring $(M, F)$ exists. 
For $i\in\set{1, \ldots, k-1}$,
if $x_i, y_i$ are in $F$, then $z_i$ must be in $M$, which forces $x_{i+1}, y_{i+1}$ to be in $F$ since $D$ $3$-fundamental gadgets have been added to $z_i$. 
%$c(z_i)=1$.
We know $x_1, y_1$ are in $F$ because $D+1$ $3$-fundamental gadgets have been added to each of $x_1$ and $y_1$, 
%of their capacities, 
so $x_k, y_k$ are forced to be in $F$.
Yet $z_k$ is also forced to be in $F$ since $D+1$ $3$-fundamental gadgets have been added to $z_k$,
%$c(z_k)=0$, 
so $x_k, y_k, z_k$ is a $3$-cycle in $F$.
Therefore $H^{s,\Delta}_k(D)$ has no $\dFF$-coloring. 

\begin{figure}
\centering
\begin{tikzpicture}
[scale=1,auto=left, 
blacknode/.style={circle,draw,fill=black,minimum size = 6pt,inner sep=0pt}, 
graynode/.style={circle,draw,fill=gray!30,minimum size = 6pt,inner sep=0pt}
]
 \begin{scope}[xshift=0cm]\figSimpleDegrees\end{scope}
\end{tikzpicture}
\caption{Sharpness example for \Cref{thm:simple-degreeSc}.}
\label{fig-simple-degrees}
\end{figure}

\section{Premise}
\label{sec-premise}

We generalize the prior concept of $\dFF$-coloring and $\eFF$-coloring to weighted graphs, which will make our arguments easier.
For a graph $G$, let $w:V(G) \rightarrow \{1, \ldots, D+2\}$ denote a weighting of the vertices.
For a weighted (multi-)graph $G$, we say that a vertex partition $(M, F)$ of $V(G)$ is an \emph{$\dFF$-coloring} if $G[M]$ and $G[F]$ are forests and for each $v \in M$ we have $w(v) + |N(v) \cap M| \leq D+1$.
We consider multi-edges as a 2-cycle, and therefore multi-edges cannot exist in $G[M]$ or $G[F]$.
Similarly, 
%We say that a vertex partition of a (multi-)graph $G$ with parts $M$ and $F$ 
$(M, F)$ is an \emph{$\eFF$-coloring} if $G[M]$ and $G[F]$ are forests and each component $C$ of $G[M]$ satisfies $\sum_{v \in C} w(v) \leq D+1$.
Note that both definitions reduce to the original definitions when $w \equiv 1$.

Observe that if $w(v) = D+2$, then in any $\dFF$-coloring or $\eFF$-coloring of $G$ it must be that $v \in F$ (a fact that we will use many times below), so it does not make sense to allow larger weights.
We define the \emph{capacity} of a vertex as $c(v) = D+2 - w(v)$.
While this notation is redundant, some of our future statements can be expressed using capacity without reference to $D$, and therefore they are easier to read.
For example, the above definition of a $\dFF$-coloring translates to $|N[v] \cap M| \leq c(v)$ for each $v \in M$ (where $N[v] = \{v\} \cup N(u)$).

Our proofs of Theorems \ref{thm:multi-degreeSc}, \ref{thm:simple-degreeSc}, \ref{thm:multi-orderSc} all begin by considering some minimum counterexample $G$.
Some of our lemmas that describe the structure of $G$ use the exact same argument for each of Theorem \ref{thm:multi-degreeSc}, Theorem \ref{thm:simple-degreeSc}, and Theorem \ref{thm:multi-orderSc}.
Instead of repeating those words, we will condense this manuscript by presenting those general arguments in this section.
First, observe that since $G$ is a minimum counterexample, $G$ is ($\dFF$- or $\eFF$-)critical and therefore connected.
Next we establish a condition on the minimum degree.

\begin{lemma}\label{2-ve}
If $v \in V(G)$, then $d(v) \geq 2$.
Moreover, if $d(v) = 2$, then $c(v) = 0$.
\end{lemma}
\begin{proof}
By criticality, $G-v$ has a ($\dFF$- or $\eFF$-)coloring $(M, F)$.
If $d_F(v)\leq 1$, then we can extend the coloring to $G$ by adding $v$ to $F$.
As $G$ is critical, the coloring does not extend, so $v$ has at least two neighbors (with multiplicity) in $F$.
This proves the first part of the lemma.
If $d(v) = 2$, then $v$ has no neighbors in $M$.
Because $(M \cup \{v\}, F)$ is not a ($\dFF$- or $\eFF$-)coloring of $G$, it follows that $c(v) = 0$.
\end{proof}

Our proofs of Theorems \ref{thm:multi-degreeSc}, \ref{thm:simple-degreeSc}, \ref{thm:multi-orderSc} will involve manipulation of a potential function.
A potential function $\rho$ represents a weighting of the vertices and edges, from which a weighting of vertex subsets is inherited.  
That is, for $A \subseteq V(G)$, we define $\rho(A) = \sum_{v \in A} \rho(v) - \sum_{e \in E(G[A])} \rho(e)$.

\begin{defn}
The specific potential functions that we will use are the following.
\begin{enumerate}
    \item We define vertex weighting $\mdrho(u)=2D+1+2c(u)$ and edge weighting $\mdrho(e) = 2(D+1)$.  
    \item We define vertex weighting $\sdrho(u)=3D+2+3c(u)$ and edge weighting $\sdrho(e) = 3(D+1)$.  
    \item For odd $D$, we define $\merho \equiv \mdrho$. 
    \item For even $D$, we define vertex and edge weightings 
    \begin{equation}\label{nn1}
\merho(v)=\begin{cases}
    2D+2c(v)  & \mbox{if  $c(v)\leq \frac{D-2}{2},$}\\
 2D+2c(v)-1  & \mbox{if  $c(v)\geq D/2,$} 
\end{cases}
\quad\mbox{and}\quad
\merho(e)=
    2D+1.
\end{equation}
\end{enumerate}
    
\end{defn}

Recall that the weighted coloring reduces to the unweighted version when we set $c \equiv D+1$.
The technical version of what we will prove is the following:

\begin{theorem}\label{potential theorem}
\begin{enumerate}
    \item  For each $D\geq 1$, every $\dFF$-{critical} multigraph $G$ on at least three vertices satisfies $\mdrho(V(G)) \leq -2$; this implies Theorem \ref{thm:multi-degreeSc}.
    \item   For each $D\geq 2$, every $\dFF$-{critical} simple graph $G$  on at least three vertices satisfies $\sdrho(V(G)) \leq -3$; this implies Theorem \ref{thm:simple-degreeSc}.
    \item  For each odd $D\geq 1$, every $\eFF$-{critical} multigraph $G$  on at least three vertices satisfies $\merho(V(G)) \leq -2$; this implies part (i) of Theorem \ref{thm:multi-orderSc}.
    \item For each even $D\geq 2$, every $\eFF$-{critical} multigraph $G$  on at least three vertices satisfies $\merho(V(G)) \leq -2$; this implies part (ii) of Theorem \ref{thm:multi-orderSc}.
\end{enumerate}    
\end{theorem}

Some of our arguments require precise calculations of potential, and we will present those lemmas in future sections.
However, there are common features among $\mdrho,\sdrho$, and $\merho$ that we can prove. 
%so we can prove some things here before dealing with the specific functions. 
Let $\rho \in \{\mdrho, \sdrho, \merho\}$ denote a general potential function that we will establish some remarks about.
%In this section, 
{ Also, everywhere below} {\em a coloring} will refer to a $\dFF$-{coloring} or a $\eFF$-{coloring}, as appropriate.
Let $\rho_0$ denote the potential of a vertex with capacity $0$, let $\rho_+$ denote the potential of a vertex with maximum capacity, and let $\rho_e$ denote the potential of an edge.
Let $\rho_*$ denote the maximum potential of a critical graph according to Theorem \ref{potential theorem} (so $\rho_* = -3$ if $\rho = \sdrho$ and $\rho_* = -2$ otherwise).
As $G$ is a counterexample, we have $\rho(V(G)) \geq \rho_* + 1$.
Let $\alpha$ denote the coefficient of the capacity term in the potential of a vertex (so $\alpha=2$ for multigraphs and $\alpha=3$ for simple graphs).

\begin{fact}\label{relative potentials}
For any $\rho \in \{\mdrho, \sdrho, \merho\}$, the following are true by the definitions of the functions: 
\begin{enumerate}[(i)]%{itemize}
    \item $\rho_e = \rho_+ - \rho_0$.
    \item $\rho_0 - \rho_e = -1$.
    \item $2 \rho_e - \rho_+ = 1$.
    \item $\rho_0/\alpha < D+1$.
    \item $\alpha = -\rho_*$.
\end{enumerate}%{itemize}
\end{fact}

The terms in the potential function are inspired by the extremal graphs presented in Section \ref{constructions section}.
For simple graphs, each additional weight on a vertex after the first weight corresponds to an implied pendant $K_4$.
For multigraphs, each additional weight on a vertex after the first weight corresponds to an implied pendant $K_3$ with a multi-edge between the non-cut vertices.
We use weights so that these gadgets are only implied; 
this will simplify our arguments in regards to $G$ being a minimum counterexample.
We set $\alpha$ to be the overall affect on potential from those gadgets; the fact that $\alpha = -\rho_*$ is a coincidence.
There is another gadget that is not as obvious.
When $D$ is even and we consider $\eFF$-colorings, we can increase the weight of a vertex $x$ by $D/2+2$ by adding a $K_4$ with vertex set $\{x',x_1,x_2,x_3\}$ where $w(x') = 1$ and $w(x_i) = D/2+1$ for each $i$, and adding edge $xx'$.
This construction is more efficient in the sense $2D+3 < \alpha(D/2+2)$, which explains the more complicated formula for $\merho$ when $D$ is even.

Our proof of \Cref{potential theorem}, which implies Theorems \ref{thm:multi-degreeSc}, \ref{thm:simple-degreeSc}, \ref{thm:multi-orderSc}, 
will conclude by discharging.
Observe that $ \rho(V(G)) = \sum_{u \in V(G)} \left(\rho(u) - d(u)\rho_e / 2\right)$.
We denote the initial charge of a vertex as $ch(u) = \rho(u) - d(u)\rho_e / 2$.
If $d(u) \geq 4$, then by Fact~\ref{relative potentials}(iii), $ch(u) \leq \rho_+ - 2\rho_e = -1$.
By Lemma \ref{2-ve} we have that $d(u) \geq 2$.
Moreover, if $d(u) = 2$, then $c(u) = 0$, so $ch(u) = \rho_0 - \rho_e = -1$.
So to prove Theorem \ref{potential theorem}, we will only need to prove that $3$-vertices have small capacity.
In some cases we will also need to show that the neighbors of $3$-vertices have small charge.

We now prove a lemma on sets with minimum potential. It will be used to prove the gap lemmas in \Cref{sec-gaplems}. 
For $A\subset V(G)$, let the {\em boundary} of $A$, denoted $\Gamma(A)$, be the set of vertices in $A$ that have neighbors outside of $A$.

\begin{lemma}\label{prelim gap lemma}
Suppose $\rho \in \{\mdrho, \sdrho, \merho\}$ and
%, and a $\Xi$-coloring is
%an    $\dFF$-{coloring} when $\rho \in \{\mdrho, \sdrho\}$ and an   $\eFF$-{coloring} when $\rho =\merho$.
a set $ S \subsetneq V(G)$ with $2 \leq |S| \leq |V(G)|-1$
 has minimum potential among such sets. If $\rho(S) \leq \rho_0 + (2 + \rho_*)$, then

(i) $|S|\neq |V(G)|-1$,

(ii) for each $v\in V(G)-S$, $|E(v,S)|\leq 1$,

(iii)
 for each $x\in \Gamma(S) $  and any coloring  $(M', F')$ of $G[S]$, $x\in M'$.
\end{lemma}

\begin{proof}  If $|S| = |V(G)|-1$
and $\{v\} = V(G)-S$, then
 $\rho_* + 1 \leq \rho(G) = \rho(S) + \rho(v) - d(v)\rho_e$.
By Lemma \ref{2-ve}, either $d(v) = 2$ and $\rho(v) = \rho_0$, or $d(v) \geq 3$.
In the first case, Fact~\ref{relative potentials}~(ii) implies
$$\rho(S) \geq \rho_* + 1 + 2\rho_e - \rho_0 = \rho_* + \rho_0+3
%\rho_e + 2
. $$
In the second case, Fact~\ref{relative potentials}~(ii),~(iii) implies
$$\rho(S) \geq \rho_* + 1 + 3\rho_e - \rho_+ = \rho_* + 
\rho_0+3. 
%\rho_e + 2.
 $$
This proves (i).

\medskip
If there is $u\in V(G)-S$ with $\setsize{E(u, S)}\geq 2$,
% If $|N(u) \cap S| \geq 2$, 
then by  Fact \ref{relative potentials}~(iii)
 $$\rho(S + u) = \rho(S) + \rho(u) - \rho_e \setsize{E(u, S)} \leq \rho(S) + \rho_+ - 2\rho_e < \rho(S).$$
% $$\rho(S + u) = \rho(S) + \rho(u) - \rho_e|N(u) \cap S| \leq \rho(S) + \rho_+ - 2\rho_e < \rho(S).$$
Since by (i), $|S+u|\leq |V(G)|-1$, this contradicts the choice of $S$.
Thus (ii) is proved.

\medskip
Now, suppose there is a coloring  $(M', F')$ of $G[S]$ where
$x\in F'\cap \Gamma(S)$. Let  $y$ be a neighbor of $x$ in $V(G)-S$. Construct a graph $G'$ from $G-S$ as follows:
\begin{enumerate}[(a)]
    \item Identify all vertices of $F'$ to a single vertex $w$ and remove all loops.  Give $w$ capacity $0$. 
    \item Reduce the capacity of each vertex in $N(M')-S -y$ to $0$.
    \item Remove all vertices in $M'$.
\end{enumerate}
By (ii), $w$ is not incident to multiple edges, so if $G$ were a simple graph then $G'$ is also simple.
%Observe that if $u \notin A$ has two neighbors in $F'$, then $wu$ is a multi-edge.

If
 $G'$ has a coloring  $(M'', F'')$, consider  coloring $(M, F)$ where 
  $M=M' \cup M''$ and $F= F' \cup F''-w$. By construction, 
  neither $M$ nor $F$ has a cycle. Also no
   $(M,M)$-edges connect $S$ with $V(G)-S$. Thus $(M, F)$ is a coloring of $G$, a contradiction. 

So, $G'$ has no coloring. By the minimality of $G$, there is $B\subseteq V(G')$ with
$\rho_{G'}(B)\leq \rho_*$. If $w\in B$, then
$$\rho_{G}((B-w)\cup S)\leq \rho_{G'}(B)+\rho_{G}(S)-\rho_{G'}(\{w\})\leq \rho_*+\rho_0 + (2 + \rho_*)-\rho_0\leq \rho_*,$$
a contradiction. 

Otherwise $w\notin B$. In this case, if $B\cup S\neq V(G)$, then
$$\rho_{G}(B\cup S)\leq \rho_{G'}(B)+\rho_{G}(S)\leq \rho_*+\rho_{G}(S),$$
contradicting the choice of $S$, and if $B\cup S= V(G)$, then since the potential of $y$ was not decreased in $G'$,
$$\rho_{G}(B\cup S)\leq \rho_{G'}(B)+\rho_{G}(S)-\rho(xy)\leq \rho_*+
\rho_0 + (2 + \rho_*)-\rho_e\leq \rho_*+(\rho_0-\rho_e)
.$$
By  Fact \ref{relative potentials}~(ii), this is a contradiction, which
 proves (iii).
\end{proof}

\section{Gap Lemmas}\label{sec-gaplems}

In this section, we prove the so called ``Gap Lemmas'' for the types of problems we investigate. 
Roughly speaking, the gap lemmas provide  lower bounds on the potential  of vertex subsets. 

We will first prove a gap lemma that works for all (multi)graphs and all colorings we consider, except when $D$ is even for multigraphs and $\eFF$-colorings.
When $D$ is even for multigraphs and $\eFF$-colorings, we first prove a (weak) gap lemma, and then prove a strong gap lemma that strengthens the (weak) gap lemma.
We separate the case when $D$ is even for multigraphs and $\eFF$-colorings into a subsection.

\begin{lemma}\label{common gap lemma}
Suppose $\rho \in \{\mdrho, \sdrho, \merho\}$, and if $\rho = \merho$, then $D$ is odd.
In particular, $\rho_e = \alpha(D+1)$.
If $S \subsetneq V(G)$ with $2 \leq |S| \leq |V(G)|-1$, then $\rho(S) > \rho_0 + (2 + \rho_*)$.
\end{lemma}
\begin{proof}
Let $S$ be a proper subset of $V(G)$ with smallest potential among those with at least two vertices.  Suppose $\rho(S)\leq \rho_0+(2+\rho_*)\leq \rho_0$.

We define $s = \max\{0, \lceil \rho(S)/\alpha \rceil\}$. Since $\rho(S)\leq \rho_0$,
 Fact \ref{relative potentials}~(iv) yields $s \leq D+1$.
Fix an arbitrary edge $xy$ such that $x \in S$ and $y \notin S$, which must exist because $G$ is connected.
Let $G'$ be $G[S]$ after reducing the capacity of $x$ by $\min\{s, c(x)\}$.
By Fact \ref{relative potentials}~(v), $\rho(S) - s \alpha > -\alpha = \rho_*$.
By the choice of $S$, every subset of $V(G')$ has potential strictly greater than $\rho_*$.
As $G$ is a minimum counterexample, and $G'$ is smaller than $G$, the graph $G'$ has no critical subgraph.
Therefore there exists a coloring $(M', F')$ of $G'$. By Lemma~\ref{prelim gap lemma}(iii), $\Gamma(S)\subseteq M'$; in particular, $x\in M'$.
% with parts $M'$ and $F'$.

We construct a graph $G''$ by modifying $G$ based on $M'$ and $F'$ as follows:
\begin{itemize}
   % \item Identify all vertices in $F'$ to a single vertex $w$. % Call this vertex $w$.  
  %  Give $w$ capacity $0$. 
    %Set the capacity of $w$ to be $0$.
    \item For each $u \in N(M')-S-y$ reduce the capacity of $u$ to $0$. 
    \item Reduce the capacity of $y$ by $\min\{c(y), D+1-s\}$.
    \item Remove all vertices in $S$.
\end{itemize}
%By the above claim, $w$ is not in a multi-edge.
%So if $G$ is a simple graph, then $G''$ is a simple graph as well.
As $x\in S$, $G''$ is smaller than $G$.

We claim that if $G''$ has a coloring $(M'', F'')$, then $(M' \cup M'', F' \cup F'' )$ forms a coloring of $G$.
Since $\Gamma(S)\cap F'=\emptyset$, there is no cycle in $G[F' \cup F''] $.
The only possible edge between $M'$ and $M''$ is $xy$ (by Lemma~\ref{prelim gap lemma}(ii), the unique neighbor of $y$ in $S$ is $x$), and therefore there is no cycle in $G[M' \cup M'']$.
So the only issue left to consider is whether edge $xy$ 
%$xy \in M' \cup M''$ 
violates the additional restriction on forests.
But %for
 if $\set{x,y} \subseteq M' \cup M''$, then $x$ and $y$ must have had positive capacity in $G'$ and $G''$, respectively. 
Thus, in the construction of $G'$, the capacity of $x$ was reduced by $s$, and in the construction of $G''$ the capacity of $y$ was reduced by $D+1-s$.
\begin{itemize}
    \item For an $\dFF$-{coloring}, $x$ having positive capacity in $G'$ after reducing its capacity by $s$ implies $s < D+1$, and therefore the capacity of $y$ was reduced by at least $1$.  
    Symmetrically, $y$ having positive capacity in $G''$ after  reducing its capacity by $D+1-s$ implies $D+1-s < D+1$, and therefore the capacity on $x$ was reduced by at least $1$.  This means that $|N[x] \cap (M' \cup M'')| \leq c(x)$ and $|N[y] \cap (M' \cup M'')| \leq c(y)$, as desired.
    \item For an $\eFF$-{coloring}, let $C'$ be the component of $M'$ containing $x$ and let $C''$ be the component of $M''$ containing $y$.  By the above, we have that $\sum_{u \in C'} c(u) \leq D+1-s$ and $\sum_{v \in C''} c(v) \leq D+1-(D+1-s) = s$.  As $C' \cup C''$ is a component induced by $M' \cup M''$, we get our desired bound $\sum_{u \in C' \cup C''}c(u) \leq (D+1-s) + s = D+1$.
\end{itemize}
This proves the claim, which contradicts the choice of $G$.
Therefore $G''$ contains a critical subgraph on vertex set $W$.
By the minimality of $G$,  $\rho_{G''}(W) \leq \rho_*$.

We will show that $\rho(S \cup W) \leq \min\{\rho(S)-1, \rho_*\}$.
Assuming the inequality is proven, if $S \cup W$ is a proper subset of $V(G)$, then this contradicts the choice of $S$.
If $S \cup W = V(G)$, then this contradicts the choice of $G$.
So in either case we get a contradiction.

We have 
\begin{equation}\label{potential for common gap lemma}
 \rho_G(S \cup W) \leq \rho(S) + \rho_{G''}(W)  + (\rho_{G}(W) - \rho_{G''}(W)) - \rho_e|E(M',V(G)-S)|.
\end{equation}
The term $\rho_{G}(W) - \rho_{G''}(W)$ accounts for setting the capacity to $0$ of the neighbors of $M'$ in $W-y$ in the construction of $G''$; it follows that $\rho_{G}(W) - \rho_{G''}(W) \leq |E(M',V(G)-S)|(\rho_+ - \rho_0)$.
By Fact \ref{relative potentials}~(i), $\rho_e = \rho_+ - \rho_0$, and therefore 
$$ (\rho_{G}(W) - \rho_{G''}(W)) - \rho_e|E(M',V(G)-S)| \leq 0. $$
Thus the above becomes 
$$ \rho_G(S \cup (W)) \leq \rho(S) + \rho_*.$$
As  $\rho_* < 0$, we have proven $ \rho_G(S \cup W) \leq \rho(S)-1$.
We now will show that $\rho_G(S \cup (W)) \leq \rho_*$.

%If $\rho(S) \geq  \rho_0 + (2 + \rho_*)$, then we are done, so assume below that $\rho(S) \leq \rho_0 + (2 + \rho_*) \leq \rho_0$.
Recall $\rho(S) \leq \rho_0$.
So the above becomes 
$$ \rho_G(S \cup W) \leq \rho_0 + \rho_*.$$

 Since 
 we reduced the capacity of $y$ by at most $D+1-s$ instead of $D+1$, the maximum possible. 
Therefore our above bound improves to 
$$\rho_{G}(W) - \rho_{G''}(W) \leq |E(M',V(G)-S)|(\rho_+ - \rho_0) - \alpha(D+1 - (D+1-s))$$
So we can improve \eqref{potential for common gap lemma} to say 
$$\rho_G(S \cup W) \leq \rho(S) + \rho_* - \alpha \lceil \rho(S)/\alpha \rceil \leq \rho_* . $$
This contradiction proves the lemma.
\end{proof}

\subsection{$\merho$ with even $D$}

% For the extent of this subsection, 
Throughout this subsection, 
assume that $D$ is even.
We will establish an analogue of Lemma \ref{common gap lemma} for $\merho$ under these circumstances.
Since we are working with a specific potential function, we will not use notation like $\rho_e, \rho_+$.
However, some of our arguments will be similar, and it might help to keep as reference that $\rho_0 = 2D$, $\rho_+ = 4D+1$, $\rho_e = 2D+1$, $\rho_* = -2$, and $\alpha = 2$.
%For $A\subset V(G)$, let {\em the boundary}, $\Gamma(A)$, denote the set of vertices in $A$ that have neighbors outside of $A$.

\begin{lemma}[Weak gap lemma for $\merho$ and even $D$]\label{gap lemma for even D} 
If $S\subsetneq V(G)$ with $2\leq |S|\leq |V(G)|-1$, then $\merho(S) \geq  D$. 
Moreover, if $\merho(S) =  D$, then
 $|E(S,V(G)-S)|=1$ and $\merho(V(G)-S) = D$.
\end{lemma}

\begin{proof}
Assume that $\merho(S)$ is minimum  among all  $S\subsetneq V(G)$ with $2\leq |S|\leq |V(G)|-1$ and that
 $\merho(S) \leq  D$. Then all hypotheses of Lemma~\ref{prelim gap lemma} hold for $S$, so the conclusions for \Cref{prelim gap lemma} follow as well.

%We first prove that
%\begin{equation}\label{inM}
%    \parbox{14cm}{\em for each $x\in \Gamma(A) $  and any $\eFF$-%coloring  $(M', F')$ of $G[A]$, $x\in M'$.}
%\end{equation}

Our first claim is:
\begin{equation}\label{2edges}
   |E_G(S,V(G)-S)|=1.
    %\parbox{14cm}{\em for each $x\in \Gamma(A) $  and any $\eFF$-coloring  $(M', F')$ of $G[A]$, $x\in M'$.}
\end{equation}

Indeed, suppose for $i\in\set{1,2}$, $x_i\in S,$ $y_i\in V(G)-S$, and edge $e_i\in E(G)$ connects $x_i$ with $y_i$.
 Let $G'=G-e_2$, where $c(y_2)$ is decreased to $0$. If $G'$ has an $\eFF$-coloring $(M',F')$, then since this contains  an $\eFF$-coloring of $G[S]$,
 by \Cref{prelim gap lemma}~(iii), $x_2\in M'$. 
 On the other hand, since  $c_{G'}(y_2)=0$, $y_2\in F'$, so $(M',F')$ is an $\eFF$-coloring of $G$, a contradiction.

 Thus $G'$ has no $\eFF$-coloring. By the minimality of $G$, there is 
 $B\subseteq V(G')$ with $\merho_{G'}(B)\leq -2$. By the construction of $G'$,
 $y_2\in B$ and $\merho_{G}(B)\leq\merho_{G'}(B)+(2D+1) \leq 2D-1$.
 If $B\cap S\neq\emptyset,$ then by the choice of $S$, $\merho_{G}(S)\leq 
 \merho_{G}(S\cap B)$ and so by submodularity of the potential function,
 $$\merho_{G}(S\cup B)\leq \merho_{G}( B)-\merho_{G}(x_2y_2)\leq (2D-1)-(2D+1)=-2, $$
a contradiction. 

If $B\cap S=\emptyset$ and $y_1\notin B$, then
$$\merho_{G}(S\cup B)\leq \merho_{G}( B)+\merho_{G}( S)-\merho_{G}(x_2y_2)\leq (2D-1)+\merho_{G}( S)
-(2D+1)=-2+\merho_{G}( S),
 $$
contradicting the choice of $S$. Finally if $B\cap S=\emptyset$ and $y_1\in B$, then
$$\merho_{G}(S\cup B)\leq \merho_{G}( B)+\merho_{G}( S)-\merho_{G}(\{x_1y_1,x_2y_2\})\leq (2D-1)+\merho_{G}( S)
 -2(2D+1)\leq-D-3.
 $$
This contradiction proves~\eqref{2edges}.

So, we may assume that $xy$ is the unique edge connecting $S$ with $V(G)-S$ where $x\in S$.
Suppose $\merho(S)\leq D-1.$
Let $s = \max\{0,\min\{c(x), \lceil \merho(S)/2 \rceil\}\}$, so $s\in\{0,\ldots,D/2\}$.
Let $H$ be $G[S]$ with the capacity of $x$ reduced by $s$.
By choice of $S$ and the minimality of $G$, $H$ has an $\eFF$-coloring $(M', F')$.

Obtain $G''$ from $G-S$ by reducing the capacity of $y$ by 
    $\min\{c(y),D+1-s\}$.
If 
 $G''$ has an $\eFF$-coloring  $(M'', F'')$, then consider
 a coloring $(M, F)$ where   $M=M' \cup M''$ and $F= F' \cup F''$. 
  By construction, 
  neither $M$ nor $F$ has a cycle. Furthermore,
  the only $(M,M)$-edge connecting $S$ with $V(G)-S$ can be $xy$. 
  If $xy$ is not an $(M, M)$-edge, then $(M, F)$ is an $\eFF$-coloring of $G$, a contradiction. 

  Suppose $xy$ is an $(M,M)$-edge and consider the component $M_0$ containing $x$ and $y$. Let 
  $M_x$ (respectively, $M_y$) be the component of $M_0-xy$ containing $x$ (respectively, $y$).
Since $x\in M'$, $s<c(x)$, and
 $W_G(V(M_x))\leq D+1-s$. Similarly, since $y\in M''$, $c_{G''}(y)\geq 1$.
So by the definition of $G''$, $D+1-s\leq c_G(y)$ and $W_G(V(M_y))\leq D+1-( D+1-s)=s$. 
Thus, $W_G(V(M_0))\leq (D+1-s)+s=  D+1$ and hence $(M,F)$ is an $\eFF$-coloring of $G$, a contradiction. 

Therefore, $G''$ has no $\eFF$-coloring, and by the minimality of $G$, $G''$ contains
a vertex subset $B$ with $\merho_{G''}(B)\leq -2$. By construction, $y\in B$
and $c_{G''}(y)>0$. Then
$c_G(y)-c_{G''}(y)=D+1-s$. So, if $s\leq \frac{D}{2}-1$, then
$c_{G''}(y)\leq s\leq \frac{D}{2}-1$ and by~\eqref{nn1}, 
\begin{equation}\label{other}
 \merho_G(B)-\merho_{G''}(B)= \merho_G(y)- \merho_{G''}(y)\leq 2(D+1-s)-1.   
\end{equation}
Thus $$\merho_G(S\cup B)\leq \merho_G(S)+\merho_{G''}(B)+(2D-2s+1) -\merho_{G}(xy)$$
\begin{equation}\label{D-1}
    \leq 2s+(-2)+(2D-2s+1)-(2D+1)=-2,
\end{equation}
 a contradiction. 
 Otherwise, $s=\frac{D}{2}$, so 
 %If $s\geq \frac{D}{2},$ then since $\merho(A)\leq D-1$, we have $s= \frac{D}{2}$ and $\merho(A)= D-1$. In this case, 
 instead of~\eqref{other}, we have a slightly weaker inequality 
 $ \merho_G(B)-\merho_{G''}(B)\leq 2(D+1-s)$
 %but
 and instead of $\merho_G(S)\leq 2s$ we have $\merho_G(S)\leq 2s-1$ since $\merho(S)\leq D-1$. 
 Thus, as in~\eqref{D-1}, we still get $\merho_G(S\cup B)\leq -2$.

This contradiction proves the first part of the lemma and that $|E(S,V(G)-S)|=1$ when $\merho(S)=D$.
It remains to show that $\merho(V(G)-S) = D$ when $\merho(S) = D$. Indeed, if 
$\merho(V(G)-S) \leq D-1$, then $\merho(V(G))\leq D+(D-1)-(2D+1)=-2$, a contradiction. 
On the other hand, if
$\merho(V(G)-S) \geq D+1$, then $\merho_{G''}(V(G''))\geq (D+1)-2s\geq 1$. So, in the argument above
$B\neq V(G)-S$ and $\merho_G(S\cup B)< \merho_G(S)$ contradicting the choice of $S$. 
\end{proof}

\begin{lemma}\label{cut-edge} Suppose $G$ has a cut edge $e=xy$, and for $z\in \{x,y\},$ $G_z=(V_z,E_z)$ is the component of $G-e$ containing $z$.
Then for $z\in \{x,y\},$

(a) in each $\eFF$-coloring $(M_z,F_z)$ of $G_z$, $z\in M_z$ and the component
$T_z$ of $G[M_z]$ containing $z$ satisfies $W(T_z)\geq \frac{D+2}{2}$;

(b) $\merho(V_z)=D$;

(c) $w(z)=1$;

(d) $d(z)\geq 4$.
\end{lemma}

\begin{proof} If say, $G_x$ has an $\eFF$-coloring $(M_x,F_x)$ such that
$x\in F_x$, then we take any $\eFF$-coloring $(M_y,F_y)$ of $G_y$, and 
$(M_x\cup M_y,F_x\cup F_y)$ is  an $\eFF$-coloring of $G$. 

Suppose now that $G_x$ has an $\eFF$-coloring $(M_x,F_x)$ such that
$x\in M_x$ and the component
$T_x$ of $G[M_x]$ containing $x$ satisfies $W(T_x)\leq \frac{D}{2}$.
Obtain $G'_y$ from $G_y$ by decreasing $c(y)$ by $\min\{c(y),\frac{D}{2}\}$. If $G'_y$ has an $\eFF$-coloring $(M_y,F_y)$, then by construction either $y\in F_y$ or the component
$T_y$ of $G[M_y]$ containing $y$ satisfies $W(T_y)\leq D+1- \frac{D}{2}$.
In both cases, $(M_x\cup M_y,F_x\cup F_y)$ is  an $\eFF$-coloring of $G$. Thus $G'_y$ has no $\eFF$-coloring. Then by the minimality of $G$, there is
$B\subseteq V_y$ with $\merho_{G'_y}(B)\leq -2$. So either $y\notin B$
and $\merho_{G_y}(B)\leq -2$ or $y\in B$ and $\merho_{G_y}(B)\leq -2+2\frac{D}{2}=D-2$.
In both cases we get a contradiction to \Cref{gap lemma for even D}. This proves (a).

\medskip
If $\merho(V_x)\geq D+1$, then by \Cref{gap lemma for even D} every  subset of $V_x$ containing $x$ has potential at least $D+1$.
In this case, 
the graph $G'_x$  obtained from $G_x$ by decreasing $c(x)$ by $\min\{c(x),1+\frac{D}{2}\}$ satisfies $\merho_{G'_x}(A)\geq -1$ for every $A\subseteq V_x$ containing $x$. Thus, $G'_x$ has an $\eFF$-coloring $(M_x,F_x)$. By (a),
$x\in M_x$ and $W(T_x)\geq \frac{D+2}{2}$. But this contradicts the definition of $G'_x$. Thus, (b) holds.

\medskip
If $w(x)\geq 2,$ then by~\eqref{nn1}, in the graph $G'_x$ defined in the paragraph above, the potential of $x$ (and hence of any subset of $V_x$  containing $x$) decreases by at most $D+1$. So we get a contradiction as in the paragraph above. Thus (c) holds.

\medskip
 Suppose $d(x)\leq 3$. 
Consider $G''=G_x-x$, which is a subgraph of $G$. 
So $G''$ has has an $\eFF$-coloring $(M'',F'')$.
If adding $x$ to $F''$ gives an $\eFF$-coloring of $G_x$, then this contradicts (a) since $x\not\in M_x$.
Otherwise, $x$ has two neighbors in $F''$, so 
we can add $x$ to $M''$.
Yet, since $w(x)= 1$ by (c), we obtain $W(\set{x})=1<\frac{D+2}{2}$, again contradicting (a).
 \end{proof}

\begin{lemma}\label{unique}
$G$ has at most one cut edge.
\end{lemma}

\begin{proof} Suppose $G$ has  cut edges $e_1=x_1y_1$ and $e_2=x_2y_2$.
Suppose further  that $x_1$ and $x_2$ are in the same component $G_x=(V_x,E_x)$ of $G-e_1-e_2$ and that the other two components are
$G_{y_1}$ and $G_{y_2}$ where $G_{y_i}$ contains $y_j$ for $j\in\set{1,2}$.

By \Cref{cut-edge}, for $j\in\set{1,2}$ $\merho(V_{y_j})=D$ and 
in each $\eFF$-coloring $(M_{y_j},F_{y_j})$ of $G_{y_j}$, ${y_j}\in M_{y_j}$ and the component
$T_{y_j}$ of $G[M_{y_j}]$ containing ${y_j}$ satisfies $W(T_{y_j})\geq \frac{D+2}{2}$. But then the multigraph  obtained from $G-V_x$ by adding edge $y_1y_2$ is a smaller counterexample to our theorem.
\end{proof}

\begin{lemma}[Strong gap lemma for $\merho$ and even $D$]\label{stronger gap lemma for even D}  
% Let $D$ be even. 
If $A\subsetneq V(G)$ with $2\leq |A|\leq |V(G)|-1$ and $|E(A,V(G)-A)|\geq 2$, then $\merho(A) \geq  2D+1$.
\end{lemma}

\begin{proof} 
By \Cref{unique}, either $G$
is $2$-edge-connected or there is a unique cut edge $e_0=x_0y_0$. In the latter case,
if the components of $G-e_0$ are $G_{x_0}=(V_{x_0},E_{x_0})$ and $G_{y_0}=(V_{y_0},E_{y_0})$, then
$\merho(V_{x_0})=\merho(V_{y_0})=D$ by \Cref{cut-edge}. We will use the following observation:
\begin{equation}\label{x0y0}
    \mbox{\em If $V_{x_0}\subseteq  U \subset V(G)$ and $\merho( U )\leq 2D$, then
    $ U =V_{x_0}$ or $y_0\in  U $.
 }
\end{equation}
Indeed, if $V_{x_0}\subsetneq  U $ and $y_0\notin  U $, then
$ U' := U -V_{x_0}\neq \emptyset$ and $\merho( U' )\leq 2D-D=D$. This contradicts either \Cref{gap lemma for even D} or \Cref{unique}.
Thus~\eqref{x0y0} holds.

Assume that $\merho(A)$ is minimum  among all  $A\subsetneq V(G)$ with $2\leq |A|\leq |V(G)|-1$ and $|E(A,V(G)-A)|\geq 2$ and that
 $\merho(A) \leq  2D$. 

We now prove that
\begin{equation}\label{inM2}
    \parbox{14cm}{\em for each $x\in \Gamma(A) $  and any $\eFF$-coloring  $(M', F')$ of $G[A]$, $x\in M'$.}
\end{equation}
Indeed, suppose $x\in F'\cap \Gamma(A)$ and $y$ is a neighbor of $x$ in $V(G)-A$. Construct a graph $G'$ from $G-A$ as follows:
\begin{enumerate}[(a)]
    \item Identify all vertices of $F'$ to a single vertex $w$ and remove all loops.  Give $w$ capacity $0$. 
    \item Reduce the capacity of each vertex in $N(M')-A-y$ to $0$. 
    \item Remove all vertices in $M'$.
\end{enumerate}
Observe that if $u \notin A$ has two neighbors in $F'$, then $wu$ is a multi-edge.

If $G'$ has an $\eFF$-coloring  $(M'', F'')$, consider a coloring $(M, F)$ where  $M=M' \cup M''$ and $F= F' \cup F''-w$. 
  By construction, 
  neither $M$ nor $F$ has a cycle. Also no
   $(M,M)$-edges connect $A$ with $V(G)-A$. Thus $(M, F)$ is an $\eFF$-coloring of $G$, a contradiction.

So, $G'$ has no $\eFF$-coloring. By the minimality of $G$, there is $B\subseteq V(G')$ with
$\merho_{G'}(B)\leq -2$. If $w\in B$, then
$$\merho_{G}((B-w)\cup A)\leq \merho_{G'}(B)+\merho_{G}(A)-\merho_{G'}(\{w\})\leq -2+2D-2D= -2,$$
a contradiction. 
 Note that reducing the capacity of a vertex to $0$ changes the potential by at most the potential of an edge.
Otherwise $w\notin B$. 

If $y \in B$, then we did not account for edge $xy$,  so 
$$\merho_{G}((B-w)\cup A)\leq \merho_{G'}(B)+\merho_{G}(A) - \merho(xy) \leq -2 + 2D - (2D+1) =-3, $$
a contradiction.
 Thus, $y\notin B$.

If $B\cup A\neq V(G)$, then
$$\merho_{G}(B\cup A)\leq \merho_{G'}(B)+\merho_{G}(A)\leq -2+\merho_{G}(A),$$
 which by choice of $A$ implies that $|E(B \cup A, V(G)-B \cup A)| = 1$.
In particular, $E(B \cup A, V(G)-B \cup A)= \{xy\}$, as $x \in A$ and $y \notin B \cup A$.
So $B \cup A = V_x$ and $xy$ is a cut edge.
By Lemma \ref{cut-edge}, $\rho(V_y) = D$, so 
$$\rho(A \cup V_y) = \rho(A) + \rho(V_y) - \rho(xy) \leq 2D + D - (2D+1) = D-1,$$
which contradicts the  weak gap lemma, \Cref{gap lemma for even D}.
%contradicting the choice of $A$, 
If $B\cup A= V(G)$, then since the potential of $y$ was not decreased,
$$\merho_{G}(B\cup A)\leq \merho_{G'}(B)+\merho_{G}(A)-\merho(xy)\leq -2+2D-(2D+1)=-3.$$
This contradiction proves~\eqref{inM2}.

We now prove the lemma. 
Suppose for $i\in\set{1,2}$, $x_i\in A,$ $y_i\in V(G)-A$, and edge $e_i\in E(G)$ connects $x_i$ with $y_i$.
 Let $G'=G-e_2$, where $c(y_2)$ is decreased to $0$. If $G'$ has an $\eFF$-coloring $(M',F')$, then since this contains  an $\eFF$-coloring of $G[A]$,
and by~\eqref{inM2}, $x_2\in M'$. 
 On the other hand, since $c_{G'}(y_2)=0$, $y_2\in F'$,  so $(M',F')$ is an $\eFF$-coloring of $G$, a contradiction.

So, $G'$ has no $\eFF$-coloring. By the minimality of $G$, there is $B\subseteq V(G')$ with $\merho_{G'}(B)\leq -2$. 
By the construction of $G'$,
 $y_2\in B$ and $\merho_{G}(B)\leq\merho_{G'}(B)+(2D+1)\leq 2D-1$.

 \medskip
{\bf Case 1:}  $B\cap A\neq\emptyset$.
If $\merho_{G}(A)\leq 
 \merho_{G}(A\cap B)$, then
 $$\merho_{G}(A\cup B)\leq \merho_{G}( B) -\merho_{G}(x_2y_2)\leq
 (2D-1)-(2D+1)=-2,
 $$
a contradiction. So suppose $\merho_{G}(A)> 
 \merho_{G}(A\cap B)$.  This implies  $A\cap B$ cannot be a single vertex since $\merho_G(A\cap B)< 2D$.
Then by the choice of $A$, $A\cap B\in \{V_{x_0},V_{y_0}\}$, say $A\cap B= V_{x_0}$. 
 Since $y_2\in B$, $ B\neq V_{x_0}$, so~\eqref{x0y0} implies $y_0\in B$. 
 Since $|E(A,V(G)-A)|\geq 2$, $ A\neq V_{x_0}$, so again~\eqref{x0y0} implies $y_0\in A$.
 This contradicts $A\cap B=V_{x_0}$.

 \medskip
{\bf Case 2:}  $B\cap A=\emptyset$.
If $y_1\notin B$, then
$$\merho_{G}(A\cup B)\leq \merho_{G}( B)+\merho_{G}( A)-\merho_{G}(x_2y_2)\leq (2D-1)+\merho_{G}( A)
-(2D+1)=-2+\merho_{G}( A).
 $$
Since $A\cup B\neq V(G)$, by the choice of $A$, we have $A\cup B\in \{V_{x_0},V_{y_0}\}$, say $A\cup B= V_{x_0}$. If $x_0\in A$, then $\merho(A\cup V_{y_0})\leq \merho(A)+D-(2D+1)\leq D-1$, contradicting \Cref{gap lemma for even D}. Otherwise,
$x_0\in B$, and $\merho(B\cup V_{y_0})\leq \merho(B)+D-(2D+1)\leq D-2$, again contradicting \Cref{gap lemma for even D}.
 
Finally if $y_1\in B$, then
$$\merho_{G}(A\cup B)\leq \merho_{G}( B)+\merho_{G}( A)-\merho_{G}(\{x_1y_1,x_2y_2\})\leq (2D-1)+\merho_{G}( A)
 -2(2D+1)\leq-3, $$
 a contradiction.
\end{proof}

\section{Common Reducible Configurations}
\label{sec-common-reducibles}

We prove two lemmas that apply to all theorems we wish to prove. 

\begin{lemma}\label{3 multi edge}
    If $d(u) = 3$ and $c(u) \geq 1$, then $u$ is not in a multi-edge.
\end{lemma}
\begin{proof}
By way of contradiction, suppose that $ux$ is a multi-edge, $uy$ is a simple edge, and $u$ has no other neighbors.
Let $G'$ be $G-u$ with the capacity of $y$ reduced by $\min\{c(y),w(u)\}$.
Observe that any coloring $(M, F)$ of $G'$ can be extended to $G$ by adding $u$ to the  color class not containing $x$ (recall that $w(u) \geq 1$, so if $y \in M$, then its capacity has been reduced by at least $1$).
So $G'$ contains a critical subgraph on vertex set $B$ and $y\in B$.
By the minimality of $G$, we have that $\rho_{G'}(B) \leq \rho_*$.

{
By Fact~\ref{relative potentials}~(i) and~(ii), $\rho_+ - \rho_0 = \rho_0 + 1$.
Observe that 
$$\rho_G(B) \leq  \rho_* + (\rho_+ - \rho_0)  =  \rho_0 + (2 + \rho_*) - 1. $$
Also, as $u$ is not in $G'$, it follows that $u$ is not in $B$, and so $B$ is a proper vertex subset of $G$.
If $\rho \neq \merho$ or $D$ is odd, then this contradicts Lemma \ref{common gap lemma}.

So assume $\rho = \merho$ and $D$ is even.
% As $G$ is a critical graph, $y \in B$.
Lemma \ref{cut-edge} implies that $uy$ is not a cut edge because $d(u) = 3$.
So the above bound on potential contradicts Lemma~\ref{stronger gap lemma for even D}.}
% 
% \iffalse
% If $x \in B$, then 
% \begin{eqnarray*}
%  \rho(B \cup \{u\}) &\leq& \rho_{G''}(B) + (\rho_G(B) - \rho_{G'}(B)) + \rho(u) - 3\rho_e \\
%         & \leq & \rho_* +  (\rho_+ - \rho_0) + \rho_+ - 3\rho_e \\
%         & = & \rho_* - 1.
% \end{eqnarray*}
% If $B \cup \{u\} = V(G)$, this contradicts the choice of $G$.
% If $B \cup \{u\} \neq V(G)$, then this contradicts Lemma \ref{common gap lemma} or Lemma \ref{gap lemma for even D} .
% So we may assume $x \notin B$, and so $B \cup \{u\}$ is a proper subset of $V(G)$.
% 
% {\MY Let $\delta_u = 1$ if $\rho = \merho$, $D$ is even, and $c(u) \leq \frac{D-2}{2}$; and let $\delta_u = 0$ otherwise.
% We have that $\rho(u) = \rho_+ - \alpha (w(u) - 1) + \delta_u$. }
% We can see that 
% \begin{eqnarray*}
%  \rho(B \cup \{u\}) &\leq& \rho_{G''}(B) + (\rho_G(B) - \rho_{G'}(B)) + \rho(u) - \rho_e \\
%         & \leq & \rho_* + \alpha w(u) + (\rho_+ - \alpha (w(u)-1) + \delta_u) - \rho_e \\
%         & \leq & \rho_0 + \delta_u. 
% \end{eqnarray*}
% As $xu$ is a parallel edge, it is not a cut-edge.
% This contradicts Lemma \ref{common gap lemma} or Lemma \ref{stronger gap lemma for even D}.
% 
% \fi
\end{proof}

\begin{lemma}\label{D+2}
If $xy \in E(G)$ and $d(x) = d(y) = 3$, then $w(x) + w(y) \geq D+2$.
\end{lemma}
\begin{proof}
By way of contradiction, suppose that $w(x) + w(y) \leq D + 1$.
By symmetry, suppose $w(x) \leq (D+1)/2$, so $c(x)\geq 2$. 
Note that by~\Cref{3 multi edge}, $x$ is not in a multi-edge.
Let $N(x) = \{y, u, v\}$ and $N(y) = \{x,a,b\}$.
Let $G'$ be $G-x,y$ where for $z \in \{u,v\}$ the capacity of $z$ is decreased by $\min\{c(z), w(x)\}$.

We claim that if $G'$ has a coloring $(M', F')$, then we can color $G$ by the following steps.
\begin{enumerate}
    \item If  $\{a,b\} \subseteq F'$, then add $y$ to $M'$.  Otherwise add $y$ to $F'$.
    \item After $y$ has been colored, add $x$ to the color class that appears the least among its three neighbors.
\end{enumerate}
Because for each of the two above steps we do not add more than one edge whose endpoints are the same color, at no point we created a cycle in $M'$ or $F'$.
The first step does not expand any components in $M'$.
So if the additional restriction on the first forest is violated, then it happened when $x$ was added to $M'$.
If $x$ was added to $M'$, then by construction it had at most one neighbor $z$ that was also in $M'$.
If $z = y$, then the component in $M'$ that contains $x$ is induced on $\{x,y\}$, and by assumption it has weight at most $D+1$.
If $z \in \{u,v\}$, then the change in capacity in the construction of $G'$ offsetted the additional weight caused by adding $x$ to the component in $M'$.
This proves the claim.
The claim contradicts the choice of $G$, so $G'$ contains a critical subgraph on vertex set $B$.

By the minimality of $G$, we have that $\rho_{G'}(B) \leq \rho_*$.
Let $t = |B \cap \{u,v\}|$.
As $G$ is critical, $t \geq 1$.
If $t=1$, then let $B' = B$ and observe $\rho_G(B) = \rho_{G'}(B) + (\rho_G(B) - \rho_{G'}(B)) \leq \rho_* + \alpha w(x) < \rho_0-1$.
If $t=2$, then let $B' = B \cup \{x\}$ and observe (using Fact \ref{relative potentials} and $c(x) \geq D+ 2 - (D+1)/2 > D/2$).
\begin{eqnarray*}
    \rho_G(B') & = & \rho_{G'}(B) + (\rho_G(B) - \rho_{G'}(B)) + \rho(x) - 2\rho_e \\
            & \leq & \rho_* + 2 \alpha w(x) + (\rho_+ - \alpha (w(x)-1)) - 2\rho_e \\
            & = & \rho_* + \alpha (w(x) + 1) - 1 \\
            & = & \alpha w(x) - 1 < \rho_0-1.
\end{eqnarray*}
In either case we find a vertex subset $B'$ with potential at most $\rho_0$ that is proper (because it does not contain $y$) and an edge $e'$ incident with $x$ that has exactly one endpoint in $B'$.
\begin{itemize}
    \item If $D$ is even and $\rho = \merho$, then because $d(x) = 3$, Lemma \ref{cut-edge} implies that $e'$
    %the edge with one endpoint in $B'$ 
    is not a cut edge.  The bound $\rho_G(B') < \rho_0 -1$ contradicts Lemma \ref{stronger gap lemma for even D}.
    \item Otherwise the bound $\rho_G(B') < \rho_0-1$ contradicts Lemma \ref{common gap lemma}.
\end{itemize}
This proves the lemma.    
\end{proof}

\section{$\dFF$-colorings}\label{sec-dFF-coloring}

In this section we finish the proofs to Theorem \ref{thm:multi-degreeSc} when $D\geq 2$ and \Cref{thm:simple-degreeSc}.
In particular, this section only deals with $\dFF$-colorings, and so $\rho \in \{\mdrho, \sdrho\}$.
Observe that for $\rho \in \{\mdrho, \sdrho\}$ we have that $\rho_e = \alpha(D+1)$.

\begin{lemma}\label{max degree d3}
    If $D \geq 2$ and $d(u) = 3$, then $c(u) \leq 1$.    
\end{lemma}
\begin{proof}
Suppose $N(u) = \{x,y,z\}$.
By way of contradiction, suppose $c(u) \geq 2$.
By Lemma \ref{3 multi edge}, $x,y,z$ are distinct.
Let $G'$ be $G - u$ where each $v \in \{x,y,z\}$ has its capacity reduced by $\min\{1,c(z)\}$.
An $\dFF$-coloring of $G'$ can be extended to an $\dFF$-coloring of $G$ by adding $u$ to the color class that appears the least on its neighbors.
So $G'$ contains an $\dFF$-critical subgraph $G''$ on vertex set $B$.
By the minimality of $G$, we have that $\rho_{G''}(B) \leq \rho_*$.

Let $t = |B \cap N(u)|$.
As $G$ is $\dFF$-critical, $t \geq 1$.
If $t \leq 2$, then $\rho_G(B) \leq \rho_* + 2 \alpha = \alpha < \rho_0$.
As $u \notin B$, this contradicts Lemma \ref{common gap lemma}.
If $t=3$, then 
\begin{eqnarray*}
  \rho_G(B \cup \{u\}) &\leq& \rho_{G''}(B) + (\rho_{G}(B) - \rho_{G''}(B)) + \rho(u) - 3\rho_e \\
                & \leq & \rho_* + 3 \alpha + \rho_+ - 3\rho_e \\
                & = & \rho_* + 3 \alpha - 1 - \rho_e \\
                & = & \rho_* + (2- D)\alpha - 1.  
\end{eqnarray*}
Because $D \geq 2$, we have $\rho(B \cup \{u\}) <  \rho_*$.
This contradicts the choice of $G$ if $B \cup \{u\} = V(G)$, and it contradicts Lemma \ref{common gap lemma} otherwise.
\end{proof}

\begin{proof}[Proof of Theorem \ref{thm:multi-degreeSc} when $D \geq 2$ and Theorem \ref{thm:simple-degreeSc}]
Recall that $\rho(G) = \sum_{u \in V(G)}ch(u)$ and that $ch(u) = \rho(u) - d(u)\rho_e/2$.
Also, $ch(u) \leq -1$ unless $d(u) = 3$.
By Lemma \ref{max degree d3}, if $d(u) = 3$, then 
$$ch(u) \leq \rho_0 + \alpha - 3\rho_e/2 = \alpha - \rho_e/2 - 1 = \alpha(1-(D+1)/2)-1 < -1.$$
So $\rho(G) \leq \sum_{u \in V(G)}(-1) \leq \rho_*$, which proves part 1 of~\Cref{potential theorem} when $D \geq 2$ and part 2 of Theorem \ref{potential theorem}.
\end{proof}

An $\ndFF{1}$-coloring is equivalent to an $\neFF{1}$-coloring---partitioning into a matching and a forest.
Recall, when $D=1$ we have $\mdrho = \merho$.
The proof of Theorem \ref{thm:multi-degreeSc} when $D = 1$ will follow from the proof of Theorem \ref{thm:multi-orderSc}~(i), presented in Section \ref{edge odd section}.

\section{$\eFF$-colorings}\label{sec-eFF-coloring}

In this section we prove \Cref{thm:multi-orderSc}, so we only deal with multigraphs and potential function $\merho$.
The proof for odd $D$ required a bit more work, which we put in a subsection. 

\begin{lemma}\label{edges d3}
    If $d(v) = 3$, then $w(v) \geq (D+1)/2$.
\end{lemma}
\begin{proof}
    By way of contradiction, suppose $d(v) = 3$ and $w(v) \leq D/2$.
    By Lemma \ref{3 multi edge}, $v$ is not in a multi-edge.
    Let $N(v) = \{x,y,z\}$.

    Construct $G'$ from $G-v$ by decreasing the capacity of each $u \in \{x,y,z\}$ by $\min\{c(u), w(v)\}$.
    If $G'$ has an $\eFF$-coloring, then it can be extended to $G$ by giving $u$ the color that appears the least among its neighbors.
    This contradicts the choice of $G$, so $G'$ contains an $\eFF$-critical subgraph on vertex set $B$.
    By the minimality of $G$, $\rho_{G'}(B) \leq \rho_*$.

    Observe that as $d(v) = 3$, by Lemma \ref{cut-edge} none of the edges incident with $v$ are cut edges when $D$ is even. 

    Let $t = |B \cap N(v)|$.
    As $G$ is critical, $t \geq 1$.
    If $t \leq 2$, then 
    $$\rho_G(B) \leq \rho_* + t \alpha w(v) \leq \rho_* + D \alpha \leq \rho_* + \rho_e - 1 = \rho_* + \rho_0 < \rho_0.$$
    This contradicts Lemma \ref{common gap lemma} or \ref{stronger gap lemma for even D}.

    So $t = 3$.
    Observe that because $w(v) \leq D/2$, we have that $\rho(v) = \rho_+ - \alpha (w(v)-1)\leq \rho_+-\alpha w(v)$.
    It follows that 
    \begin{eqnarray*}
        \rho_{G}(B \cup \{v\})  & = & \rho_{G'}(B) + (\rho_G(B) - \rho_{G'}(B)) + \rho(v) - 3 \rho_e \\
                & \leq & \rho_* + 3 \alpha w(v) + (\rho_+ - \alpha w(v)) - 3 \rho_e \\
                & = & \rho_* + 2 \alpha w(v) - 1 - \rho_e \\
                & \leq & \rho_* + D \alpha - 1 - \rho_e < \rho_*.
    \end{eqnarray*}
    If $B \cup \{v\} = V(G)$, then this contradicts the choice of $G$.
    Otherwise this contradicts Lemma \ref{common gap lemma} or \ref{gap lemma for even D}.
\end{proof}

\begin{proof}[Proof of Theorem \ref{thm:multi-orderSc}.(ii)]
Assume $D$ is even, and therefore $D \geq 2$.
Recall that $\rho(G) = \sum_{u \in V(G)}ch(u)$ and that $ch(u) = \rho(u) - d(u)\rho_e/2$.
Also, $ch(u) \leq -1$ unless $d(u) = 3$.
By Lemma \ref{edges d3}, if $d(u) = 3$, then $w(u) \geq D/2 + 1$, which implies $c(u) \leq D+2 - (D/2+1) = D/2+1$.
So if $d(u) = 3$ and $c(u) \in \{D/2, D/2+1\}$, then 
$$ch(u) \leq (2D + 2(D/2+1)-1) - 3(2D+1)/2 = -1/2.$$
Otherwise if $d(u) = 3$ and $c(u) \leq D/2-1$, then 
$$ch(u) \leq (2D + 2(D/2-1)) - 3(2D+1)/2 = -7/2.$$
So $\rho(G) \leq \sum_{u \in V(G)}(-1/2) \leq -3/2$.
As $\rho(G)$ is integral, we have $\rho(G) \leq \rho_*$, which proves part 4 of Theorem \ref{potential theorem}.
\end{proof}

\subsection{Odd $D$}\label{edge odd section}

In this section we consider $\eFF$-colorings on multigraphs for odd $D$.
Let $T$ denote the set of $3$-vertices with weight $(D+1)/2$; by Lemma \ref{edges d3} this is the smallest possible weight for a $3$-vertex.
By Lemma \ref{3 multi edge}, no vertex in $T$ is incident with a multi-edge.
By Lemma \ref{D+2}, $T$ is an independent set.

\begin{lemma}\label{multi-4vx-3333o}
$G$ has no $4$-vertex $u$ with $w(u)\leq \frac{D+1}{2}$ such that $N(u) \subseteq T$.
\end{lemma}
\begin{proof} 
Suppose a $4$-vertex $u$ with $w(u)\leq{\frac{D+1}{2}}$ is adjacent to four $3$-vertices $v_1,\ldots,v_4$ in $T$.
Note that each $v_i$ has weight ${\frac{D+1}{2}}$ and is not in a multi-edge.  
For $1\leq i\leq 4$, let $N(v_i)=\{u,x_i,y_i\}$.
By assumption, $x_i\neq y_i$ for $1\leq i\leq 4$.
Recall that $ \merho(u)\leq 4D+3$ and $ \merho(v_i)= 3D+4.$

Obtain $G'$ from $G-\{u,v_1,v_2,v_3\}$ by adding edges $x_iy_i$ for $1\leq i\leq 3$.

\medskip
{\bf Case 1:} $G'$ has an $\eFF$-coloring $(M, F)$. 
If $v_4\in F$, then $(M+u, F\cup\set{v_1,v_2,v_3})$ is an $\eFF$-coloring of all of $G$, which is a contradiction. 
Suppose $v_4\in M$.
If $\set{x_4, y_4}\subseteq F$, then $(M+u, F\cup\set{v_1,v_2,v_3})$ is an $\eFF$-coloring of all of $G$, since $w(u)+w(v_4)\leq 2\frac{D+1}{2}= D+1$.
Otherwise, $(M+u-v_4, F\cup\set{v_1,v_2,v_3,v_4})$ is an $\eFF$-coloring of all of $G$.

\medskip
{\bf Case 2:} $G'$ has no $\eFF$-coloring. By the minimality of $G$, $G'$ has a vertex subset $A$ with $ \merho_{G'}(A)\leq -2$. Since $G'[A]$ is not a subgraph of $G$, it must contain at least one of the added edges. 
Let $t=\setsize{\set{i:\set{x_i,y_i}\subseteq A}}$.
By symmetry, suppose $x_1y_1, \ldots, x_ty_t \subset A$.

{\bf Case 2.1}: $t=1$. 
Then $ \merho(A \cup \{v_1\}) \leq \merho_{G'}(A)+ (2D+2)+(3D+4)-2(2D+2)\leq D ,$ which contradicts Lemma \ref{common gap lemma}.  

{\bf Case 2.2}: $t=2$. 
Then $$ \merho(A \cup \{u, v_1, v_2\})\leq \merho_{G'}(A)+
    2(2D+2)+(4D+3)+2(3D+4)-6(2D+2)\leq 2D+1,$$ which contradicts Lemma \ref{common gap lemma}.  

{\bf Case 2.3}: $t=3$. 
If $v_4\not\in A$, then 
$$ \merho(A \cup \{v_1, v_1, v_3, u\})\leq \merho_{G'}(A)+ 3(2D+2)+(4D+3)+3(3D+4)-9(2D+2)\leq D+1,$$ 
which contradicts Lemma \ref{common gap lemma}.  
If $v_4\in A$, then 
$$ \merho(A\cup \{v_1, v_1, v_3, u\} )\leq \merho_{G'}(A)+
    3(2D+2)+(4D+3)+3(3D+4)-10(2D+2)\leq -D-1\leq -2,
$$ which contradicts the choice of $G$. 
\end{proof}

\begin{lemma}\label{2-veo}
If $v$ is a $2$-vertex, then $v$ is in a multi-edge.
\end{lemma}
\begin{proof} Suppose a $2$-vertex $v$ has two distinct neighbors $u_1$ and $u_2$.
Consider the multigraph $G''$ obtained from $G-v$ by adding one edge 
connecting $u_1$ with $u_2$. 
If $G''$ has an $\eFF$-coloring $(M', F')$, then we can extend it to $G$ by adding $v$ to $F'$. 
Otherwise, by the minimality of $G$, there is $A\subseteq V(G'')$ with $\merho_{G''}(A)\leq -2$.
Note that $\{u_1, u_2\}\subseteq A$. 
Observe that $\merho_G(A)\leq \merho_{G''}(A)+2D+2\leq 2D$. 
As $u \notin A$, this contradicts Lemma \ref{common gap lemma}.
\end{proof}

A consequence of Lemma \ref{2-veo} is that no vertex in $T$ is adjacent to a $2$-vertex.

\begin{proof}[Proof of Theorem \ref{thm:multi-orderSc}.(i)]

By \Cref{2-ve}, recall that $ ch(u) = (4D+5 - 2w(u))-d(u)(D+1)$ and if $d(u) \neq 3$, then $ch(u) \leq -1$.
Moreover, a direct calculation gives that if a $3$-vertex $u \notin T$, then $ch(u) \leq -1$.
If $u \in T$, then $ch(u) = 1$.

\bigskip

We use one {\bf discharging rule}:

\begin{enumerate}[({R}1)]
    \item\label{R:32o} Each vertex in $T$ sends charge $\frac{1}{3}$ to each neighbor.
\end{enumerate}

Let $ch^*(v)$ denote the final charge of a vertex $v$. 
If $v\in T$, then as $T$ is an independent set, $ch^*(v)=ch(v)-3(1/3)=0$. 
For each $v\notin T$, let $j(v)$ denote
the number of edges incident with both $v$ and a vertex in $T$. 
By (R1),
\begin{equation}\label{disch}
\mbox{\em for each $v\notin T$,
 $ch^*(v)=ch(v)+\frac{j(v)}{3}$. }
\end{equation}
% Consider possible cases.

{\bf Case 1}:
  $v\notin T$ and $d(v)\leq 3$. Here, 
  $$ ch^*(v)=ch(v)+\frac{j(v)}{3}\leq -1+\frac{j(v)}{3}\leq
 -1+\frac{d(v)}{3}\leq 0.
  $$
  Moreover, if $j(v) < 3$, then $ch^*(v) < 0$.
In particular, by Lemma~\ref{2-veo}, if $d(v)=2$, then $j(v)=0$ so  $ch^*(v)=-1$.

{\bf Case 2}: $d(v)=4$. 
If $w(v)=1$, then by \Cref{multi-4vx-3333o}, $j(v)\leq 3$, so 
\begin{equation}\label{di3o}
ch^*(v)=ch(v)+\frac{j(v)}{3}\leq -1+\frac{3}{3}=0.    
\end{equation}
%\begin{equation}\label{di3o}
% ch^*(v)=4D+5-2-4(D+1)+\frac{j(v)}{3}\leq 0.    
% \end{equation}
Moreover, if $j(v) < 3$, then $ch^*(v) < 0$.
If $w(v)\geq 2$, then
\begin{equation}\label{di4o}
ch^*(v)= ch(v)+\frac{j(v)}{3}\leq 4D+5-4-4(D+1)+\frac{j(v)}{3}\leq 
-3+\frac{4}{3}= -\frac{5}{3}.
\end{equation}
% \begin{equation}\label{di4o}
% ch^*(v)\leq ch(v)+\frac{j(v)}{3}\leq 
% -1-2+\frac{4}{3}\leq -\frac{5}{3}.
% \end{equation}

{\bf Case 3}: $d(v)\geq 5$.
A direct calculation gives 
$$ch^*(v)\leq (4D+3) - d(v)(D+1-1/3)  < -D \leq -1.$$

By the discharging above, $ch^*(v)\leq 0$ for each $v\in V(G)$. 
So, if there is a set $A\subseteq V(G)$ with $\sum_{v\in A}ch^*(v)<-1$,
 then $\sum_{v\in V(G)}ch^*(v)=\merho(V(G))<-1$, and hence $\merho(V(G))\leq -2$. 
 Thus assume that
\begin{equation}\label{dis3}
  \mbox{for each $A\subseteq V(G)$,  $\quad\sum_{v\in A}ch^*(v)\geq -1.$}  
\end{equation}

In particular $d(u)\leq 4$ for every vertex $u$. 

Suppose first that $G$ has a $2$-vertex $v$.
By Case 1 above, $ch^*(v) = -1$.
Moreover, by Lemma~\ref{2-veo}, $v$ is in a multi-edge, say $vu$.
In this case $u$ has at most $d(u)-2$ neighbors in $T$, and so $ch^*(u)<0$, which yields $ch^*(\{v,u\})<-1$.

Observe that if a $3$-vertex $v$ satisfies $w(v)\geq \frac{D+5}{2}$, then 
 $ch^*(v)\leq (4D+5)-2\frac{D+5}{2}-3(D+1)+3\frac{1}{3}=-2$.
So $G$ only contains (A) vertices in $T$, (B) $3$-vertices with weight $(D+3)/2$, and (C) $4$-vertices with weight $1$.
Let $T' = V(G) - T$ be the vertices of type (B) and (C).

Recall that by Lemma~\ref{multi-4vx-3333o},
  $j(v)\leq 3$ for each $v\in T'$. 
  So, in order for~\eqref{dis3}  to hold,  we need
\begin{equation}\label{charge-edgeso}
Q(A):=\sum_{v\in A}(3-j(v)) \leq 3\quad
   \mbox{\em for each $A \subseteq T'$.}
\end{equation}

If $G[T']$ has no cycles, then the partition $(T,T')$ is an $\eFF$-coloring of $G$, so let $C=u_1u_2\ldots u_su_1$ be an $s$-cycle in $G[T']$. 
Since $j(u_i)\leq d(u_i)-2$ for
each $u_i$, and hence $Q(V(C))\geq \sum_{i=1}^s(5-d(u_i))\geq s,$
by~\eqref{charge-edgeso}
we need $s\leq 3$. We consider three cases.

\medskip
{\bf Case 1:} $s=3$. In order to satisfy~\eqref{charge-edgeso},
we need that
$d(u_i)=4$ for all
$1\leq i\leq 3$, 
all other vertices in $T'$ have degree at most $1$ in $G[T']$, $C$ is a component in $G[T']$, and each $3$-vertex in $T'$ is isolated in $G[T']$.

Hence, the components of $G[T'-u_1]$  can be isolated vertices or $K_2$s such that both vertices of a $K_2$ are $4$-vertices. This implies that
partition $(T'-u_1,T+u_1)$ is an $\eFF$-coloring of $G$ since $T$ is independent and no vertex in $T$ is in a multi-edge.

\medskip
 {\bf Case 2:} $s=2$ and there is $u_0\in T'-\{u_1,u_2\}$ with  $j(u_0)\leq 2$. 
 Then $Q(\{u_0,u_1,u_2\})\geq 3$ with equality only if
 $j(u_0)=j(u_1)=j(u_2)=2$. So, in order to satisfy~\eqref{charge-edgeso},
%have the total charge of $T'$  at least $-1$,
we also need that
$d(u_1)=d(u_2)=4$, 
all  vertices in $T'\setminus \{u_0,u_1,u_2\}$ have degree at most $1$ in $G[T']$, $C$ is a component in $G[T']$, and each $3$-vertex in $T'-u_0$ is isolated in $G[T']$. 

Then for $i\in\set{1,2}$, each component of $G[T'-u_0-u_i]$ has at most one edge. 
%is in $\eF$.
Therefore, if $N(u_0)\cap T\neq N(u_i)\cap T$, then $G[T+u_0+u_i]$ is a forest, and hence $(T'-u_0-u_i, T+u_0+u_i)$ is an $\eFF$-coloring of $G$,
a contradiction. Otherwise, if $N(u_0)\cap T= N(u_1)\cap T= N(u_1)\cap T=
\{x_1,x_2\}$, then
$$\merho(\{x_1,x_2,u_0,u_1,u_2\})\leq 2(3D+4)+3(4D+3)-8(2D+2)=2D+1.$$
This contradicts \Cref{common gap lemma} when $|V(G)|> 5$; and $|V(G)|> 5$ since $u_0$ has a neighbor in $T'$ that is neither $u_1$ nor $u_2$. 

 \medskip
 {\bf Case 3:} $s=2$ and $j(v)= 3$
 for each $v\in T'-\{u_1,u_2\}$. In order to satisfy~\eqref{charge-edgeso},
%have the total charge of $T'$  at least $-1$,
we also need that at least one $u_i\in V(C)$, say $u_2$, has no neighbors in $T'-u_{3-i}$. Then again, each component of $G[T'-u_1]$ has at most one edge,
%is in $\eF$, 
 and hence $(T'-u_1, T+u_1)$ is an $\eFF$-coloring of $G$, a contradiction.
\end{proof}

\section{Concluding remarks}

1. As remarked in the introduction, we were not able to prove Theorem~\ref{thm:simple-degreeS} for $D=1$.  We conjecture that  every $\paren{\frac{9}{5},\frac{2}{5}}$-sparse simple graph has an $\ndFF{1}$-coloring.

2. It would be interesting to prove an analog of Theorem~\ref{thm:multi-orderS} for simple graphs.

%\bigskip
\section*{Acknowledgment}

The authors thank Daniel Cranston for  helpful discussions and comments.

\end{document}